\documentclass[11pt]{article}
\usepackage[english]{babel}
\usepackage[a4paper,top=2.2cm,bottom=2.2cm,left=1.8cm,right=1.8cm,marginparwidth=1.75cm]{geometry}
\usepackage{settings_custom}
\usepackage[sorting=nty,style=alphabetic]{biblatex}

\definecolor{ssw}{rgb}{0.1,0.45,0.1}
\newcommand{\shire}[1]{\mathcal{S} #1}
\newcommand{\mordor}[1]{\mathcal{W} #1}
\newcommand{\doom}[1]{\mathcal{T} #1}
\newcommand{\eps}{\varepsilon}
\usepackage[normalem]{ulem}

\title{On free energy barriers in Gaussian priors and failure of cold start MCMC for high-dimensional unimodal distributions}
\author{Afonso S. Bandeira, Antoine Maillard, Richard Nickl\footnote{Correspondence: nickl@maths.cam.ac.uk. RN would like to thank the Forschungsinstitut f\"ur Mathematik (FIM) at ETH Z\"urich for their hospitality during a sabbatical visit in spring 2022 where this research was initiated. }~, Sven Wang \\ \\ ETH Z\"urich, ETH Zürich, University of Cambridge, MIT}

\begin{document}

\maketitle

\abstract{We exhibit examples of high-dimensional unimodal posterior distributions arising in non-linear regression models with Gaussian process priors for which MCMC methods can take an exponential run-time to enter the regions where the bulk of the posterior measure concentrates. Our results apply to worst-case initialised (`cold start') algorithms that are local in the sense that their step-sizes cannot be too large on average. The counter-examples hold for general MCMC schemes based on gradient or random walk steps, and the theory is illustrated for Metropolis-Hastings adjusted methods such as pCN and MALA.}

\section{Introduction}\label{sec:intro}

Markov Chain Monte Carlo (MCMC) methods are the workhorse of Bayesian computation when closed formulas for estimators or probability distributions are not available. For this reason they have been central to the development and success of high-dimensional Bayesian statistics in the last decades, where one attempts to generate samples from some \textit{posterior distribution} $\Pi(\cdot|\mathrm{data})$ arising from a prior $\Pi$ on $D$-dimensional Euclidean space and the observed data vector. MCMC methods tend to perform well in a large variety of problems, are very flexible and user-friendly, and enjoy many theoretical guarantees. Under mild assumptions, they are known to converge to their stationary `target' distributions as a consequence of the ergodic theorem, albeit perhaps at a slow speed, requiring a large number of iterations to provide numerically accurate algorithms. When the target distribution is log-concave, MCMC algorithms are known to mix rapidly, even in high dimensions. But for general $D$-dimensional densities, we have only a restricted understanding of the scaling of the mixing time of Markov chains with $D$ or with the `informativeness' (sample size or noise level) of the data vector. 

A classical source of difficulty for MCMC algorithms are multi-modal distributions. When there is a deep well in the posterior density between the starting point of an MCMC algorithm and the location where the posterior is concentrated, many MCMC algorithms are known to take an exponential time --
proportional to the depth of the well -- when attempting to reach the target region, even in low-dimensional settings, see Fig.~\ref{subfig:energy_barrier} and also the discussion surrounding Proposition \ref{nubd} below. However, for distributions with a single mode and when the dimension $D$ is fixed, MCMC methods can usually be expected to perform well.

\begin{figure}[ht]
\centering
\begin{subfigure}{.45\textwidth}
  \centering
  \includegraphics[width=.9\textwidth]{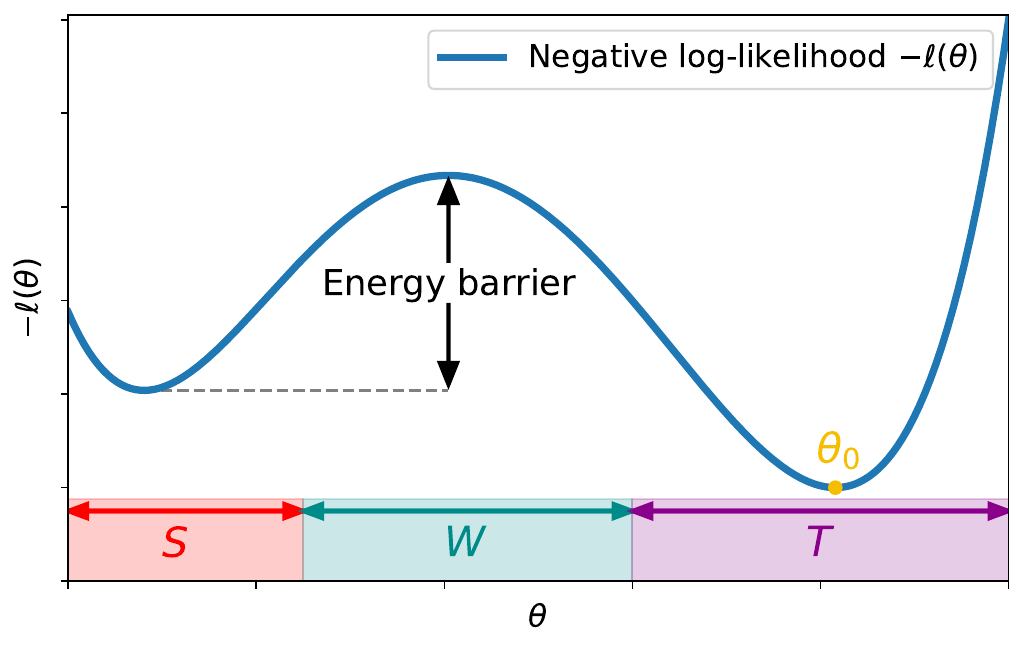}  
  \caption{In low dimensions (here $D = 1$), MCMC hardness usually arises because of a non-unimodal likelihood, creating an ``energy barrier'', even though the maximum likelihood is attained at $\theta = \theta_0$.
  The MCMC algorithm is assumed to be initialised in the set $\shire$ containing a local maximum of the likelihood.}
  \label{subfig:energy_barrier}
\end{subfigure}
\begin{subfigure}{.45\textwidth}
  \centering
  \includegraphics[width=.7\textwidth]{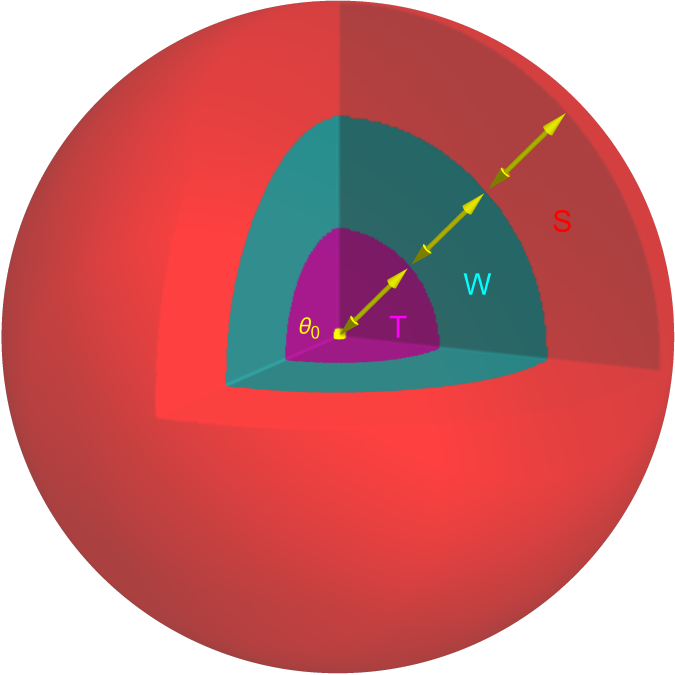}  
  \caption{Illustration of the arising of entropic (or volumetric) difficulties, here in dimension $D = 3$: the set of points close to $\theta_0$ has much less volume than the set of points far away. 
  As $D$ increases, this phenomenon is amplified: all ratios of volumes of the three sets $\doom, \mordor, \shire$ scale exponentially with $D$.}
  \label{subfig:sphere}
\end{subfigure}
\caption{
\small
Two possible sources of MCMC hardness in high dimensions: multimodal likelihoods and entropic barriers.\label{fig:origin_mcmc_hardness} }
\end{figure}

In essence this article is an attempt to explain how, in high dimensions, wells can be formed \textit{without} multi-modality of a given posterior distribution. The difficulty in this case is volumetric, also referred to as \textit{entropic}: while the target region contains most of the posterior mass, its (prior) volume is so small compared to the rest of the space that an MCMC algorithm may take an exponential time to find it, see Fig.~\ref{subfig:sphere}. 
This competition between `energy' -- here represented by the log-likelihood $\ell_N$ in the posterior distribution $d\Pi(\cdot|\mathrm{data}) = \exp \{\ell_N+ \log d \pi\}$ -- and `entropy' (related to the prior term $\pi$) has also been exploited in recent work on statistical aspects of MCMC in various high dimensional inference and statistical physics models \cite{anderson1989spin,mezard2009information,zdeborova2016statistical,arous2020algorithmic,arous2020free}. These ideas somewhat date back to the 19th century foundations of statistical mechanics \cite{gibbs1873method} and the notion of free energy, consisting of a sum of energetic and entropic contributions which the system spontaneously attempts to minimise. 
The ``MCMC-hardness'' phenomenon described above is then akin to the meta-stable behavior of thermodynamical systems, such as glasses or supercooled liquids. 
As the temperature decreases, such systems can undergo a ``first-order'' phase transition, in which a global free energy minimum (analoguous to the target region above) abruptly appears, while the system remains trapped in a suboptimal local minimum of the free energy (the starting region of the MCMC algorithm).
For the system to go to thermodynamic equilibrium it must cross an extensive free energy barrier: such a crossing requires an exponentially long time, so that the system appears equilibrated on all relevant timescales, similarly to the MCMC stuck in the starting region. Classical examples include glasses and the popular experiment of rapid freezing of supercooled water (i.e.\ water that remained liquid at negative temperatures) after introducing a perturbation. 

Inspired by recent work~\cite{arous2020algorithmic,arous2020free,bandeira2022franz}, let us illustrate some of the volumetric phenomena which are key to our results below.
We separate the parameter space into three regions (see Figures~\ref{fig:origin_mcmc_hardness} and \ref{fig:illustration_wells}), which we name by common MCMC terminology.
Firstly a \emph{starting} (or initialisation) region $\shire$, where an algorithm starts, secondly a \emph{target} region $\doom$ where both the bulk of the posterior mass and the ground truth are situated, and thirdly an intermediate \emph{Free-Entropy well}\footnote{As classical in statistical physics, we call free entropy the negative of the free energy.} $\mordor$ that separates $\shire$ from $\doom$\footnote{In a physical system these regions would correspond respectively to a region including a meta-stable state, a region including the globally stable state, and a free energy barrier.}.
In our theorems, these regions will be characterised by their Euclidean distance to the ground truth parameter $\theta_0$ generating the data.
The prior volumes of the $\epsilon$-annuli $\{\theta:r-\epsilon<\|\theta-\theta_0\|_2 \le r\}, r>0,$ closer to the ground truth are smaller than those further out as illustrated in Fig.~\ref{subfig:sphere}, and in high dimensions this effect becomes quantitative in an essential way. Specifically, the trade-off between the entropic and energetic terms can happen such that the following three statements are simultaneously true.
\begin{itemize}
    \item[$(i)$] $\doom$ contains ``almost all'' of the posterior mass.
    \item[$(ii)$]  As one gets closer to $\doom$ (and thus the ground truth $\theta_0$), the log-likelihood is strictly monotonically increasing.
    \item[$(iii)$] Yet $\shire$ still possesses exponentially more posterior mass than $\mordor$.
\end{itemize}
Using `bottleneck' arguments from Markov chain theory (Ch.7 in \cite{J03}), this means that an MCMC algorithm that starts in $\shire$ is expected to take an exponential time to visit $\mordor$.
If the step size is such that it cannot ``jump over'' $\mordor$, this also implies an exponential hitting time lower bound for reaching $\doom$.
This is illustrated in Figure~\ref{fig:illustration_wells} for an averaged version of the model described in Section~\ref{sec:spiked_tensor}.
\begin{figure}[ht]
\centering
\begin{subfigure}{.38\textwidth}
  \centering
  \includegraphics[width=\linewidth]{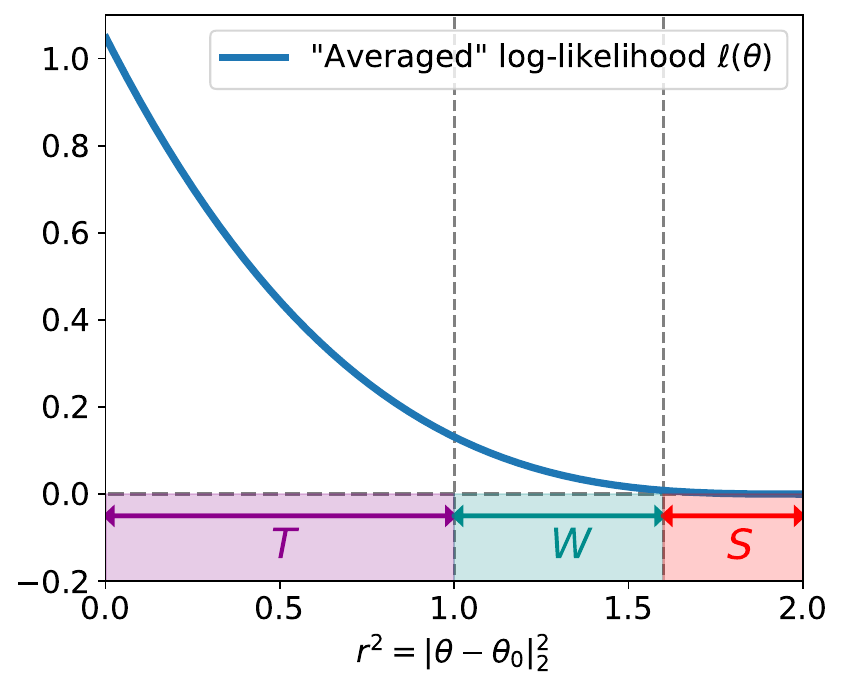}  
  \caption{}
  \label{subfig:likelihood_entropy}
\end{subfigure}
\begin{subfigure}{.6\textwidth}
  \centering
  \includegraphics[width=\linewidth]{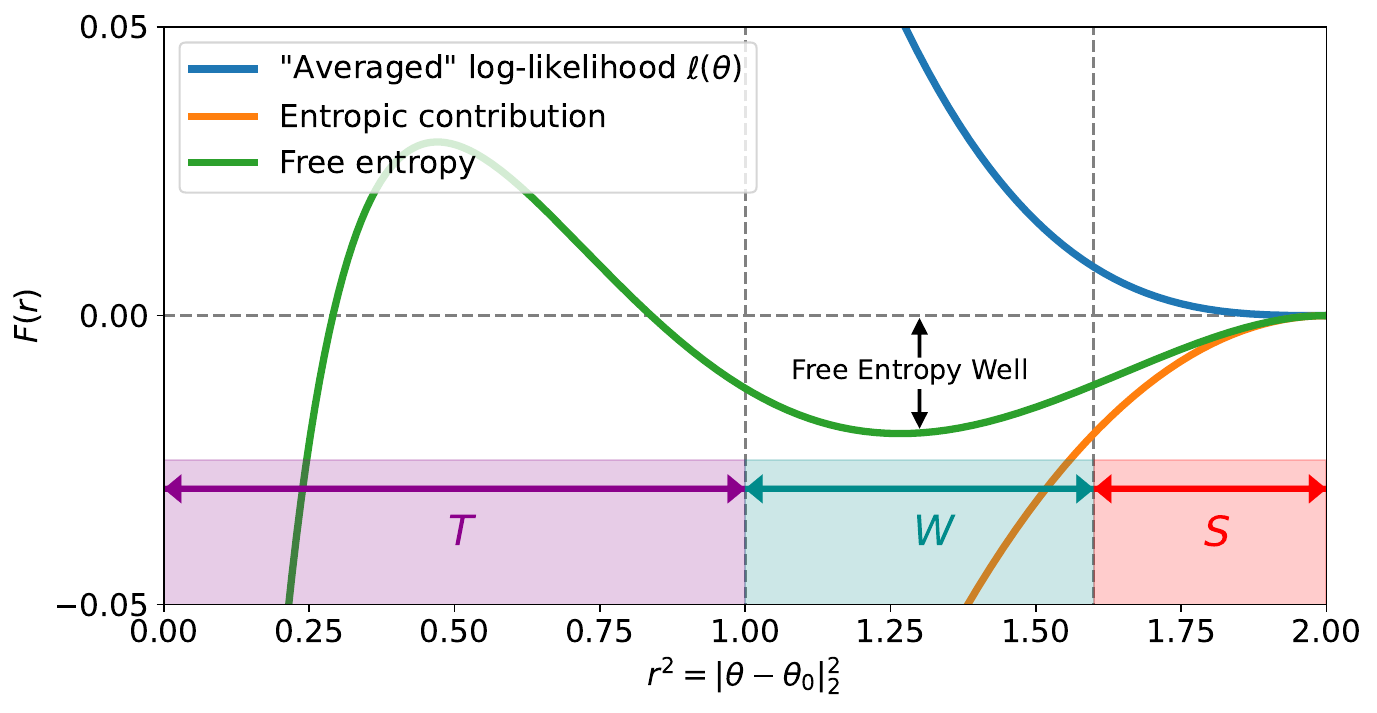}  
  \caption{}
  \label{subfig:fe_well}
\end{subfigure}
\caption{ 
\small
\label{fig:illustration_wells} Illustration of a free-energy barrier (or free entropy well) arising with a unimodal posterior. The model is an ``averaged'' version of the spiked tensor model, with log-likelihood $\ell_n(\theta) = \lambda \langle \theta, \theta_0 \rangle^3 / 2$ and uniform prior $\Pi$ on the $n$-dimensional unit sphere $\bbS^{n-1}$. $\theta_0$ is chosen arbitrarily on $\bbS^{n-1}$.
The posterior is $\rd \Pi(\theta | Y) \propto \exp\{n \ell_n(\theta)\} \rd\Pi(\theta) $, for $\theta \in \bbS^{n-1}$.
Up to a constant, the free entropy $F(r) = (1/n) \log \int \rd \Pi(\theta | Y) \delta(r - \|\theta-\theta_0\|_2)$ can be decomposed as the sum of $\ell_n(\theta)$ (that only depends on $r = \|\theta-\theta_0\|_2$) and the ``entropic'' contribution 
$(1/n) \log\int \rd \Pi(\theta) \delta(r- \|\theta-\theta_0\|_2) $.
In the figure we show $\lambda = 2.1$.
}
\end{figure}

In the situation described above, the MCMC iterates never visit the region where the posterior is statistically informative, and hence yield no better inference than a random number generator.
One could regard this as a `hardness' result about computation of posterior distributions in high dimensions by MCMC.
In this work we show that such situations can occur generically and establish hitting time lower bounds for common gradient or random walk based MCMC schemes in model problems with non-linear regression and Gaussian process priors.
Before doing this, we briefly review some important results of \cite{arous2020algorithmic} for the problem of Tensor PCA, from which the inspiration for our work was drawn.
This technique to establish lower bounds for MCMC algorithms has also recently been leveraged in \cite{arous2020free} in the context of sparse PCA, and in~\cite{bandeira2022franz} to establish connections between MCMC lower bounds and the Low Degree Method for algorithmic hardness predictions (see~\cite{kunisky2022notes} for an expository note on this technique).

When the target distribution is globally log-concave, pictures such as in Fig.~\ref{fig:illustration_wells} are ruled out (see also Remark \ref{logcn}) and polynomial-time mixing bounds have been shown for a variety of commonly used MCMC methods. While an exhaustive discussion would be beyond the scope of this paper, we mention here the seminal works \cite{D17, DM19} which were among the first to demonstrate high-dimensional mixing of discretised Langevin methods (even upon `cold-start' initialisations like the ones assumed in the present paper). In concrete non-linear regression models, polynomial-time computation guarantees were given in \cite{NW20} under a general `gradient stability' condition on the regression map which guarantees that the posterior is (with high probability) locally log-concave on a large enough region including $\theta_0$. While this condition can be expected to hold under natural injectivity hypotheses and was verified for an inverse problem with the Schr\"odinger equation in \cite{NW20}, for non-Abelian X-ray transforms in \cite{BN21}, the `Darcy flow' elliptic PDE model in \cite{N22} and for generalized linear models in \cite{A22}, all these results hinge on the existence of a suitable initialiser of the gradient MCMC scheme used. These results form part of a larger research program \cite{N20, MNP19, MNP21, MNP21b, N22} on algorithmic and statistical guarantees for Bayesian inversions methods \cite{S10} applied to problems with partial differential equations. The present article shows that the hypothesis of existence of a suitable initialiser is -- at least in principle -- essential in these results if $D/N \to \kappa>0$, and that at most `moderately' high-dimensional ($D=\smallO(N)$) MCMC implementations of Gaussian process priors may be preferable to bypass computational bottlenecks.

Our negative results apply to (worst-case initialised) Markov chains whose step-sizes cannot be too large with high  probability. As we show this includes many commonly used algorithms (such as pCN and MALA) whose dynamics are of a `local' nature. 
There are a variety of MCMC methods developed recently, such as piece-wise deterministic Markov processes, boomerang or zig-zag samplers \cite{FBPR18, BCVD18, BGKR20, WR20} which may not fall into our framework.
While we are not aware of any rigorous results that would establish polynomial hitting or mixing times of these algorithms for high-dimensional posterior distributions such as those exhibited here, it is of great interest to study whether our computational hardness barriers can be overcome by `non-local' methods. There is some empirical evidence that this may be possible. 
For instance, in the numerical simulation of models of supercooled liquids \cite{SGB22}, methods such as swap Monte Carlo \cite{GP01} have been observed to equilibrate to low-temperature distributions which were not reachable by local approaches.
Another example is given by the planted clique problem \cite{jerrum1992large}: this model is conjectured to possess a large algorithmically hard phase, 
and local Monte Carlo methods are known to fail far from the conjectured algorithmic threshold \cite{gamarnik2019landscape,angelini2021mismatching,chen2022almost}. 
On the other hand, non-local exchange Monte Carlo methods (such as parallel tempering \cite{hukushima1996exchange}), have been numerically observed to perform significantly better \cite{chiara2018parallel}.

 \section{The spiked tensor model: an illustrative example}\label{sec:spiked_tensor}

In this section, we present (a simplified version of) results obtained mostly in \cite{arous2020algorithmic}. First some notation.
For any $n \geq 1$, we denote by $\bbS^{n-1} = \{\theta \in \bbR^n: \, \|\theta\|_2 = 1 \}$ the Euclidean unit sphere in $n$ dimensions.
For $\theta,\theta' \in \bbR^n$ we denote $\theta \otimes \theta' = (\theta_i \theta'_j)_{1 \leq i,j \leq n} \in \bbR^{n^2}$ their tensor product.

\myskip
Spiked tensor estimation is a synthetic model to study tensor PCA, and corresponds to a Gaussian Additive Model with a low-rank prior.
More formally, it can be defined as follows \cite{richard2014statistical}.
\begin{definition}[Spiked tensor model]\label{def:spiked_tensor}
Let $p\geq 3$ denote the order of the tensor.
The observations $Y$ and the parameter $\theta$ are generated according to the following joint probability distribution: 
\begin{align}\label{eq:joint_Ytheta_spiked_tensor}
    \rd \bbQ(Y, \theta) = \frac{1}{(2\pi)^{n^p/2}}\exp\Big\{-\frac{1}{2} \big\|Y - \sqrt{n} \lambda \theta^{\otimes p}\big\|_2^2\Big\}\rd \Pi(\theta) \, \rd Y.
\end{align}
Here, $\rd Y$ denotes the Lebesgue measure on the space $(\bbR^n)^{\otimes p} = \bbR^{n^p}$ of $p$-tensors of size $n$.
$\Pi$ is the uniform probability measure on $\bbS^{n-1}$, and $\lambda \geq 0$ is the signal-to-noise ratio (SNR) parameter.
In particular, the posterior distribution $\Pi(\theta | Y)$ is:
\begin{align}\label{eq:def_posterior}
    \rd \Pi(\theta | Y) &=
\frac{1}{\mcZ_Y} \exp(\ell_{n,Y}(\theta)) \rd\Pi(\theta),
\end{align}
in which $\mcZ_Y$ is a normalization, and we defined the \emph{log-likelihood} (up to additive constants) as:
    \begin{align}\label{eq:def_likelihood}
        \ell_{n,Y}(\theta) = \frac12 \sqrt{n}\lambda \langle \theta^{\otimes p},Y \rangle.
    \end{align}
\end{definition}
 \noindent
 In the following, we study the model from Definition~\ref{def:spiked_tensor} via the prism of statistical inference. In particular, we will study the posterior $\Pi(\theta|Y)$ for a fixed\footnote{Note that we assume here that the statistician has access to the distribution $\Pi(\cdot | Y)$ (and in particular to $\lambda$), a setting sometimes called \emph{Bayes-optimal} in the literature.} ``data tensor'' $Y$. Since such a tensor was generated according to the marginal of \eqref{eq:joint_Ytheta_spiked_tensor}, we parameterise it as $Y = \lambda \sqrt{n} \theta_0^{\otimes p} + Z$, with $Z$ a $p$-tensor with i.i.d.\ $\mcN(0,1)$ coordinates, and $\theta_0$ a ``ground-truth'' vector uniformly-sampled in $\bbS^{n-1}$. 
 The goal of our inference task is to recover information on the low-rank perturbation $\theta_0^{\otimes p}$ (or equivalently on the vector $\theta_0$, possibly up to a global sign depending on the parity of $p$) from the posterior distribution $\Pi(\cdot | Y)$.

    Crucially, we are interested in the limit of the model of Definition~\ref{def:spiked_tensor} as $n \to \infty$. In particular, all our statements, although sometimes non-asymptotic, are to be interpreted as $n$ grows. 
    We say that an event occurs ``with high probability'' (w.h.p.) when its probability is $1-\smallO_n(1)$\footnote{Often the $\smallO_n(1)$ term will be exponentially small, but we will not require such a strong control.}.
    Moreover, by rotation invariance, all statements are uniform over $\theta_0 \in \bbS^{n-1}$, so that said probabilities only refer to the noise tensor $Z$.
Finally, 
throughout our discussion we will work with latitude intervals (or bands) on the sphere, with the North Pole taken to be $\theta_0$.
We characterise them using inner products (correlations) $\langle \theta, \theta_0\rangle$ for odd $p$, and $|\langle \theta, \theta_0\rangle|$ for even $p$ (since in this case $\theta_0$ and $-\theta_0$ are indistinguishable from the point of view of the observer).
\begin{definition}[Latitude intervals]
    \label{def:lotr}
    Assume that $p \geq 3$ is even. For $0 \leq s < t \leq 1$ we define:
\begin{itemize}
    \item $\shire_s = \left\{ \theta\in \bbS^{n-1}: |\langle \theta,\theta_0 \rangle| \leq s \right\}$,
    \item $\mordor_{s,t} = \left\{ \theta\in \bbS^{n-1}: s < |\langle \theta,\theta_0 \rangle| \leq t \right\}$,
    \item $\doom_t = \left\{ \theta\in \bbS^{n-1}: t<|\langle \theta,\theta_0 \rangle|  \right\}$.
\end{itemize}
If $p$ is odd, we define these sets similarly, replacing $|\langle \theta,\theta_0\rangle|$ by $\langle \theta,\theta_0\rangle$.
\end{definition}
\noindent
Note that these sets can also be characterised using the distance to the ground-truth, e.g.\ $\shire_s = \left\{ \theta\in \bbS^{n-1}: \min\{\|\theta-\theta_0\|_2^2,\|\theta+\theta_0\|_2^2\}| \geq 2 (1-s) \right\}$ when $p$ is even.

\subsection{Posterior contraction}\label{subsec:posterior_contraction}

\noindent
We can use uniform concentration of the likelihood to show that as $\lambda \to \infty$ (\emph{after} taking the limit $n \to \infty$) the posterior contracts 
in a region infinitesimally close to the ground-truth $\theta_0$. We first show that a region arbitrarily close to the ground truth exponentially dominates a very large starting region:
\begin{proposition}\label{prop:posterior_contraction_technical}
    For any $K > 0$ there exists $\lambda_0 > 0$ and functions $\{s(\lambda),t(\lambda)\} \in [0,1)$ such that $s(\lambda) < t(\lambda)$,
    $\{s(\lambda),t(\lambda)\} \to 1$ as $\lambda \to \infty$, and for all $\lambda \geq \lambda_0$:
    \begin{equation}
        \limsup_{n \to \infty} \frac{1}{n} \log \frac{\Pi(\shire_{s(\lambda)} |Y)}{\Pi(\doom_{t(\lambda)}|Y)} \leq - K, \hspace{0.5cm} \textrm{almost surely.}
    \end{equation}
\end{proposition}
\noindent
Posterior contraction is the content of the following result:
\begin{corollary}[Posterior contraction]\label{cor:posterior_contraction}
    There exists $\lambda_0 > 0$ 
    and a function $s(\lambda) \in [0,1)$ satisfying $s(\lambda) \to 1$ as $\lambda \to \infty$, 
    such that for all $\lambda \geq \lambda_0$:
    \begin{equation}
        \lim_{n \to \infty} \Pi[\doom_{s(\lambda)}|Y] = 1, \hspace{0.5cm} \textrm{almost surely.}
    \end{equation}
\end{corollary}
\noindent
The proofs of Proposition~\ref{prop:posterior_contraction_technical} and Corollary~\ref{cor:posterior_contraction} are given in Appendix~\ref{sec_app:proofs_spiked_tensor}.

\begin{remark}[Suboptimality of uniform bounds]\label{remark:suboptimality_uniform_bound} \normalfont
Stronger than Corollary~\ref{cor:posterior_contraction}, it is known that there exists a sharp threshold $\lambda^\star(p)$ such that for any $\lambda > \lambda^\star(p)$ the posterior mean, as well as the maximum likelihood estimator, 
sit w.h.p.\ in $\doom_{s(\lambda)}$, with $s(\lambda) > 0$, while such a statement is false for $\lambda \leq \lambda^\star(p)$\cite{perry2020statistical,lesieur2017statistical,jagannath2020statistical}.
The $\lambda_0$ given by Corollary~\ref{cor:posterior_contraction} is, on the other hand, clearly not sharp, 
because of the crude uniform bound used in the proof. 
This can easily be understood in the $p = 2$ case, corresponding to rank-one matrix estimation: uniform bounds such as the ones used here would show posterior contraction 
for $\lambda = \omega(1)$, while it is known through the celebrated BBP transition
that the maximum likelihood estimator is already correlated with the signal for any $\lambda>1$ \cite{baik2005phase}.
With more refined techniques from the study of random matrices and spin glass theory of statistical physics
it is often possible to obtain precise constants for such relevant thresholds.
\end{remark}

\subsection{Algorithmic bottleneck for MCMC}\label{subsec:algorithmic_bottleneck}

\noindent
Simple volume arguments, associated with an ingenious use of Markov's inequality due to \cite{arous2020algorithmic}
and of the rotation-invariance of the noise tensor $Z$, 
allow to get a computational hardness result for MCMC algorithms, even though the posterior contracts infinitesimally close to the ground truth as we saw in Corollary~\ref{cor:posterior_contraction}.
In the context of the spiked tensor model, these computational hardness results can be found in \cite{arous2020algorithmic} (see in particular Section~7).
We will state similar results for general non-linear regression models in Section~\ref{sec:results_gaussian_priors}: in this context we will not need to use the Markov's inequality-based technique of \cite{arous2020algorithmic}, and will solely rely on concentration arguments.

\myskip
Recall that by Section~\ref{subsec:posterior_contraction}, we can find $s(\lambda)$ such that $s(\lambda) \to 1$ as $\lambda \to \infty$
and for all $\lambda$ large enough $\Pi(\doom_{s(\lambda)}|Y) = 1 - \smallO_n(1)$.
Here, we show that escaping the ``initialisation'' region of the MCMC algorithm is hard in a large range of $\lambda$ (possibly diverging with $n$). 
In what follows, the step size of the algorithm denotes the maximal change $\norm{x^{t+1} - x^t}_2$ allowed in any iteration\footnote{
As we will detail in the following sections, see Assumption~\ref{ass:step-eta}, the statements remain true if the change is allowed to be higher than the required maximum with exponentially small probability.
}.
We first state this bottleneck result informally.
\begin{proposition}[MCMC Bottleneck, informal]\label{prop:bottleneck}
    Assume that $\lambda = \smallO(n^{(p-2)/4 + \eta})$ for all $\eta > 0$.
    Then any MCMC algorithm whose invariant distribution is $\Pi(\cdot|Y)$, and with a step size bounded by $\delta = \mathcal{O}([n\lambda^2]^{-1/p})$, will take an exponential time to get out of the 
    ``initialisation'' region.
\end{proposition}
\noindent
Note that the step size condition of Proposition~\ref{prop:bottleneck} is always meaningful, since our hypothesis on $\lambda$ implies 
$[n \lambda^2]^{-1/p} = \omega(n^{-1/2})$, and many MCMC algorithms (e.g.\ any procedure in which a number $\mcO(1)$ of coordinates of the current iterate
are changed in a single iteration) will have a step size $\mcO(n^{-1/2})$. 

\begin{remark}
    \normalfont
    The results of \cite{arous2020algorithmic} are stated when considering for the invariant distribution of the MCMC a more general ``Gibbs-type'' distribution 
    $\bbG_{\beta,Y}(\rd x) \propto e^{\beta H(x)} \rd \Pi(x)$, with $H(x) = (\sqrt{n}/2) \langle x^{\otimes p},Y \rangle$.
    The case we consider here is the ``Bayes-optimal'' $\beta = \lambda$, for which $\bbG_{\lambda,Y} = \Pi(\cdot|Y)$.
    For the general distribution $\bbG_{\beta,Y}$
    the conditions of Proposition~\ref{prop:bottleneck} become $\beta \lambda = \smallO(n^{(p-2)/2 + \eta})$ and $\delta = \mathcal{O}[(n\beta\lambda)^{-1/p}]$. 
    The authors of \cite{arous2020algorithmic} usually consider $\beta = \mathcal{O}(1)$, 
    so that they show the bottleneck under the condition $\lambda = \smallO(n^{(p-2)/2 + \eta})$.
\end{remark}

\noindent
More generally, $\lambda \ll n^{(p-2)/4}$ is conjectured to be a regime in which \emph{all} polynomial-time algorithms fail to recover $\theta_0$ \cite{richard2014statistical,wein2019kikuchi,hopkins2015tensor,hopkins2016fast,kim2017community}.
On the other hand, ``local'' methods (such as gradient-based algorithms \cite{sarao2019afraid,sarao2019passed,biroli2020iron,ben2020bounding,arous2021online}, message-passing iterations \cite{lesieur2017statistical}, or natural MCMC algorithms such as the ones of previous remark) are conjectured or known to fail in the larger range $\lambda \ll n^{(p-2)/2}$.
Proposition~\ref{prop:bottleneck} shows that ``Bayes-optimal'' MCMC algorithms fail for $\lambda \ll n^{(p-2)/4}$.
To the best of our knowledge, analysing this class of algorithms in the regime $n^{(p-2)/4} \ll \lambda \ll n^{(p-2)/2}$ is still open.

\myskip 
Let us now state formally the key ingredient behind Proposition~\ref{prop:bottleneck}.
It is a rewriting of the ``free energy wells'' result of \cite{arous2020algorithmic}.
\begin{lemma}[Bottleneck, formal]\label{lemma:bottleneck_formal}
    Assume that $\lambda = \smallO(n^{(p-2)/4 + \eta})$ for all $\eta > 0$, and
    let $\delta = \mathcal{O}([n\lambda^2]^{-1/p})$. 
    Let $r(\varepsilon) = n^{-1/2+\varepsilon}$.
    Then for any $\varepsilon > 0$ small enough, there exists $c,C > 0$ 
    such that for large enough $n$, with probability at least $1 - \exp(- c n^{2\varepsilon})$ we have:
    \begin{align}\label{eq:posterior_ratio_spiked_tensor}
        \frac{\Pi(\shire_{r(\varepsilon)}|Y)}{\Pi(\mordor_{r(\varepsilon), r(\varepsilon)+\delta}|Y)} &\geq \exp\{C n^{2\varepsilon}\}.
    \end{align}
\end{lemma}
\noindent
Note that by simple volume arguments, $\Pi(\shire_{r(\varepsilon)}) = 1 - \smallO_n(1)$, so that $\shire_{r(\varepsilon)}$ contains ``almost all'' the mass of the uniform distribution.

One can then deduce from Lemma~\ref{lemma:bottleneck_formal} hitting time lower bounds for MCMCs using a folklore bottleneck argument -- see Jerrum \cite{J03} -- that we recall here in a simplified form (see also \cite{arous2020free}, as well as Proposition~\ref{hit}, where we will detail it further along with a short proof).
\begin{proposition}\label{prop:jerrum_simplified}
   We fix any $Y$ and $n$, and let any $0 < s < t < 1$.
   Let $\theta^{(0)}, \theta^{(1)}, \cdots$ be a Markov chain on $\bbS^{n-1}$ with stationary distribution $\Pi(\cdot | Y)$, and initialised from $\theta^{(0)} \sim \Pi_{\shire_{s}}(\cdot | Y)$, the posterior distribution conditioned on $\shire_{s}$. Let $\tau_{t} = \inf\{k \in \bbN \, : \, \theta^{(k)} \in \doom_t\}$ be the hitting time of the Markov chain onto $\doom_{t}$. Then, for any $k \geq 1$, 
   \begin{align}
       \mathrm{Pr}(\tau_t \leq k) &\leq k \frac{\Pi(\mordor_{s, t}|Y)}{\Pi(\shire_{s}|Y)}.
   \end{align}
\end{proposition}

\begin{remark}[MCMC initialisation] \normalfont
    Note that Lemma~\ref{lemma:bottleneck_formal}, combined with Proposition~\ref{prop:jerrum_simplified}, shows hardness of MCMC initialised in points drawn from 
    $\Pi_{\shire_{r(\epsilon)}}(\cdot | Y)$.  In particular, it is easy to see that this implies (via the probabilistic method) the existence of such ``hard'' initialising points.
    While one might hope to show such negative results for more general initialisation, this remains an open problem. 
    On the other hand, \cite{arous2020algorithmic} shows that there exists initialisers in $\shire_{r(\epsilon)}$ for which vanilla Langevin dynamics achieve non-trivial recovery of the signal even for $\lambda = \Theta_n(1)$ (a phenomenon they call ``equatorial passes'').
\end{remark}

\section{Main results for non-linear regression with Gaussian priors}\label{sec:results_gaussian_priors}

We now turn to the main contribution of this article, which is to exhibit some of the phenomena described in Section~\ref{sec:spiked_tensor} in the context of non-linear regression models. All the theorems of this section are proven in detail in Section~\ref{proofs}.

\medskip
Consider data $Z^{(N)}=^{iid}(Y_i, X_i)_{i=1}^N$ from the random design regression model 
\begin{equation}\label{model}
Y_i = \mathscr G(\theta)(X_i)+ \varepsilon_i,~\varepsilon_i \sim \mcN(0,1),~ i=1, \dots, N,
\end{equation}
where $\mathscr G:\Theta \to L^2_\mu(\mathcal X)$ is a regression map taking values in the space $L^2(\mathcal X)=L^2_\mu(\mathcal X)$ on some bounded subset $\mathcal X$ of $\R^d$, and where the $X_i \sim^{iid} \mu$ are drawn uniformly on $\mathcal X$. For convenience, we assume that $\mathcal X$ has Lebesgue measure $\int_\mathcal X dx =1$. The law of the data $dP_\theta^N(z_1, \dots, z_N) = \prod_{i=1}^N dP_\theta(z_i)$ is a product measure on $(\mathbb R \times \mathcal X)^N$, with associated expectation operator $\mathrm E_\theta^N$. Here $\theta$ varies in some parameter space $$\Theta \subseteq \R^D,~~ \frac{D}{N} \simeq \kappa \ge 0,$$ and $\theta_0 \in \Theta$ is a `ground truth' (we could use `mis-specified' $\theta_0$ and project it onto $\Theta$). We will primarily consider the case where $\kappa>0$ and $\Theta = \R^D$, and consider high-dimensional asymptotics where $D$ (and then also $N$) diverge to infinity, even though some aspects of our proofs do not rely on these assumptions. 
We will say that events $A_N$ hold with high probability if $P^N_{\theta_0}(A_N) \to 1$ as $N \to \infty$, and we will use the same terminology later when it involves the law of some Markov chain.

\smallskip

Let $\Pi$ be a prior (Borel probability measure) on $\Theta$ so that given the data $Z^{(N)}$ the posterior measure is the `Gibbs'-type distribution
\begin{equation}\label{post}
d\Pi(\theta|Z^{(N)}) =\frac{e^{\ell_N(\theta)}d\Pi(\theta)}{\int_{\Theta} e^{\ell_N(\theta)}d\Pi(\theta)},~~~ \theta \in \Theta,
\end{equation}
where $$\ell_N(\theta) = -\frac{1}{2} \sum_{i=1}^N|Y_i - \G(\theta)(X_i)|^2, ~~\ell(\theta)=\mathrm E_{\theta_0}^N\ell_N(\theta),~~\theta \in \Theta.$$

\subsection{Hardness examples for posterior computation with Gaussian priors}

We are concerned here with the question of whether one can sample from the Gibbs' measure (\ref{post}) by MCMC algorithms. The priors will be Gaussian, so the `source' of the difficulty will arise from the log-likelihood function $\ell_N$. On the one hand, recent work (\cite{D17, DM19, NW20, BN21, N22}) has demonstrated that if $\ell_N(\theta)$ is `on average' (under $E_{\theta_0}$) log-concave, possibly only just locally near the ground truth $\theta_0$, then MCMC methods that are initialised into the area of log-concavity can mix towards $\Pi(\cdot|Z^{(N)})$ in polynomial time even in high-dimensional ($D\to \infty$) and `informative' ($N \to \infty$) settings. In absence of such structural assumptions, however, posterior computation may be intractable, and the purpose of this section is to give some concrete examples for this with choices of $\mathscr G$ that are representative for non-linear regression models.

\smallskip

We will provide lower bounds on the run-time of `worst case' initialised MCMC in settings where the average posterior surface is not \textit{globally} log-concave but still unimodal. Both the log-likelihood function and posterior density exhibit linear growth towards their modes, and the average log-likelihood is locally log-concave at $\theta_0$. In particular the Fisher information is well defined and non-singular at the ground truth. 

The computational hardness does not arise from a local optimum (`multimodality'), but from the difficulty MCMC encounters in `choosing' among many high-dimensional directions when started away from the bulk of the support of the posterior measure. That such problems occur in high dimensions is related to the probabilistic structure of the prior $\Pi$, and the manifestation of `free energy barriers' in the posterior distribution. 

In many applications of Bayesian statistics, such as in machine learning or in non-linear inverse problems with PDEs, \textit{Gaussian process priors} are commonly used for inference. To connect to such situations we illustrate the key ideas that follow with two canonical examples where the prior on $\R^D$ is the law
\begin{equation} \label{twopriors}
a)~ \theta \sim \mcN(0, \Id_D/D),~~\text{ or } b)~\theta \sim \mcN(0, \Sigma_\alpha),
\end{equation}
where $\Sigma_\alpha$ is the covariance matrix arising from the law of a $d$-dimensional Whittle-Mat\'ern-type Gaussian random field (see Section~\ref{subsubsec:def_matern_prior} for a detailed definition). 
These priors represent widely popular choices in Bayesian statistical inference \cite{RW06, GV17} and can be expected to yield consistent statistical solutions of regression problems even when $D/N \ge \kappa>0$, see \cite{vdVvZ08, GV17}. In b) we can also accommodate a further `rescaling' ($N$-dependent shrinkage) of the prior similar to what has been used in recent theory for non-linear inverse problems (\cite{MNP21}, \cite{NW20}, \cite{BN21}), see Remark \ref{rescale} for details. 

\medskip

We will present our main results for the case where the ground truth is $\theta_0=0$. This streamlines notation while also being the `hardest' case for negative results, since the priors from a) and b) are then already centred at the correct parameter.

\medskip

To formalise our results, let us define balls
\begin{equation}\label{Bs}
B_r=\{\theta \in \R^D: \|\theta\|_{\R^D} \le r\},~r>0,
\end{equation}
centred at $\theta_0=0$. We will also require the annuli
\begin{equation}
\Theta_{r, \varepsilon}= \{\theta \in \R^D: \|\theta\|_{\R^D} \in (r, r+\varepsilon) \},
\end{equation}
for $r, \varepsilon>0$ to be chosen. To connect this to the notation in the preceding sections, the sets $\Theta_{r,\eps}$ will play the role of the initialisation (or starting) region $\mathcal S$, while $B_s$ (for suitable $s$) corresponds to the target region $\mathcal T$ where the posterior mass concentrates. The `intermediate' region $\mathcal W=\Theta_{s,\eta}$ representing the `free-energy barrier' is constructed in the proofs of the theorems to follow.

\medskip

Our results hold for general Markov chains whose invariant measure equals the posterior measure~(\ref{post}), and which admit a bound on their `typical' step-sizes. As step-sizes can be random, this assumption needs to be accommodated in the probabilistic framework describing the transition probabilities of the chain. Let $\mathcal P_N(\theta,A), N \in \mathbb N,$ (for $\theta\in \R^D$ and Borel sets $A\subseteq \R^D$), denote a sequence of Markov kernels describing the Markov chain dynamics employed for the computation of the posterior distribution $\Pi(\cdot|Z^{(N)})$. Recall that a probability measure $\mu$ on $\R^D$ is called invariant for $\mathcal P_N$ if $\int_{\R^D} \mathcal P_N(\theta, A) d\mu(\theta) = \mu(A)$ for all Borel sets $A$.

\begin{assumption}\label{ass:step-eta}
Let $\mathcal P_N(\cdot,\cdot)$ be a sequence of Markov kernels satisfying the following:

\textbf{i)} $\mathcal P_N(\cdot,\cdot)$ has invariant distribution $\Pi(\cdot|Z^{(N)})$ from (\ref{post}). 

\textbf{ii)} For some fixed $c_0>0$ and for sequences $L=L_N>0$, $\eta=\eta_N>0$, with $P_0^N$-probability approaching $1$ as $N\to \infty$, 
\[ \sup_{\theta\in B_L} \mathcal P_N(\theta,\{\vartheta: \| \theta-\vartheta \|_{\R^D}\ge \eta/2 \})\le e^{-c_0N},~~~N\ge 1.  \]
\end{assumption}

This assumption states that typical steps of the Markov chain are, with high probability (both under the law of the Markov chain and the randomness of the invariant `target' measure), concentrated in an area of size $\eta/2$ around the current state $\theta$, uniformly in a ball of radius $L$ around $\theta_0=0$. 
For standard MCMC algorithms (such as pCN, MALA) whose proposal steps are based on the discretisation of some continuous-time diffusion process, such conditions can be checked, as we will show in the next subsection.

\begin{theorem}\label{a}
Let $D/N \simeq \kappa>0$, consider the posterior (\ref{post}) arising from the model (\ref{model}) and a $\mcN(0, \Id_D/D)$ prior of density $\pi$, and let $\theta_0=0$. Then there exists $\G$ and a fixed constant $s\in (0,1/3)$ for which the following statements hold true.

\textbf{i)} The expected likelihood $\ell(\theta)$ is unimodal with mode $0$, locally log-concave near $0$, radially symmetric, Lipschitz-continuous and monotonically decreasing in $\|\theta\|_{\R^D}$ on $\R^D$. 

\textbf{ii)} For any fixed $r>0$, with high probability the log-likelihood $\ell_N(\theta)$ and the posterior density $\pi(\cdot|Z^{(N)})$ are monotonically decreasing in $\|\theta\|_{\R^D}$ on the set $\{\theta: \|\theta\|_{\R^D}\ge r\}$.

\textbf{iii)} We have that $ \Pi(B_s|Z^{(N)}) \xrightarrow{N\to\infty} 1$ in probability.

\textbf{iv)} There exists $\varepsilon>0$ such that for any (sequence of) Markov kernels $\mathcal P_N$ on $\R^D$ and associated chains $(\vartheta_k:k\ge 1)$ that satisfy Assumption \ref{ass:step-eta} for some $c_0>0$, $L=1+\eps$, sequence $\eta_N\in (0,s)$ and all $N\ge 1$ large enough, we can find an initialisation point $\vartheta_0 \in \Theta_{2/3, \varepsilon}$ such that with high probability (under the law of $Z^{(N)}$ and the Markov chain), the hitting time $\tau_{B_s}$ for $\vartheta_k$ to reach $B_s$ (with $s$ as in \textbf{iii)}) is lower bounded as
	\[\tau_{B_s} \ge \exp\big(\min \{c_0,1\}N/2\big).\]

\end{theorem}

The interpretation is that despite the posterior being strictly increasing in the radial variable $\|\theta\|_{\R^D}$ (at least for $\|\theta\|_{\R^D}>r$, any $r>0$ -- note that maximisers of the posterior density may deviate from the `ground truth' $\theta_0=0$ by some asymptotically vanishing error, cf.~also Proposition \ref{mono}), MCMC algorithms started in $\Theta_{2/3, \varepsilon}$ will still take an exponential time before visiting the region $B_s$ where the posterior mass concentrates. This is true for small enough step-size independently of $D, N$. The result holds also for $\vartheta_0$ drawn from an absolutely continuous distribution on $\Theta_{2/3, \varepsilon}$ as inspection of the proof shows. Finally, we note that at the expense of more cumbersome notation, the above high probability results (and similarly in Theorem \ref{b}) could be made non-asymptotic,
in the sense that for all $\delta > 0$ all statements hold with probability at least $1 - \delta$ for all $N\ge N_0(\delta)$ large enough, where the dependency of $N_0$ on $\delta$ can be made explicit.

\medskip

For `ellipsoidally supported' $\alpha$-regular priors b), the idea is similar but the geometry of the problem changes as the prior now `prefers' low-dimensional subspaces of $\R^D$, forcing the posterior closer towards the ground truth $\theta_0=0$. We show that if the step size is small compared to a scaling $N^{-b}$ for $b>0$ determined by $\alpha$, then the same hardness phenomenon persists. Note that `small' is only `polynomially small' in $N$ and hence algorithmic hardness does not come from exponentially small step-sizes.

\begin{theorem}\label{b}
	Let $D/N \simeq \kappa>0$, consider the posterior (\ref{post}) arising from the model (\ref{model}) and a $\mcN(0,\Sigma_\alpha)$ prior of density $\pi$ for some $\alpha>d/2$, and let $\theta_0=0$. Define $b=(\alpha/d)-(1/2)>0$. Then there exists $\G$ and some fixed constant $s_b\in (0,1/2)$ for which the following statements hold true.
	
    \textbf{i)} The expected likelihood $\ell(\theta)$ is unimodal with mode $0$,
    locally log-concave near $0$, radially symmetric, Lipschitz continuous and monotonically decreasing in $\|\theta\|_{\R^D}$ on $\R^D$. 
	
	\textbf{ii)} For any fixed $r>0$, with high probability $\ell_N(\theta)$ is radially symmetric and decreasing in $\|\theta\|_{\R^D}$ on the set $\{\theta:\|\theta\|_{\R^D} \ge r N^{-b}\}$.

	\textbf{iii)} Defining $s=s_b N^{-b}$, we have $\Pi(B_s|Z^{(N)}) \xrightarrow{N\to\infty} 1$ in probability. 
	
	\textbf{iv)} There exist positive constants $\varepsilon,C>0$ and $\nu=\nu(\kappa, \alpha, d)>0$ such that for any (sequence of) Markov kernels $\mathcal P_N$ on $\R^D$ and associated chains $(\vartheta_k:k\ge 1)$ that satisfy Assumption \ref{ass:step-eta} for some $c_0>0$, $L=L_N=C\sqrt N$,  sequence $\eta=\eta_N\in (0, s_bN^{-b})$ and all $N\ge 1$ large enough, we can find an initialisation point $\vartheta_0 \in \Theta_{N^{-b}, \varepsilon N^{-b}}$ such that with high probability (under the law of $Z^{(N)}$ and the Markov chain), the hitting time $\tau_{B_s}$ for $\vartheta_k$ to reach $B_s$ is lower bounded as
		\[\tau_{B_s} \ge \exp\big(\min \{c_0,\nu\}N/2\big).\]
\end{theorem}

Again, \textbf{iv)} holds as well for $\vartheta_0$ drawn from an absolutely continuous distribution on $\Theta_{N^{-b}, (1+\varepsilon) N^{-b}}$. We also note that $\eps$ depends only on $\alpha, \kappa, d$ and the choice of $\G$ but not on any other parameters.

\begin{remark}\normalfont
As opposed to Theorem \ref{a}, due to the anisotropy of the prior density $\pi$, the posterior distribution is no longer radially symmetric in the preceding theorem, whence part \textbf{ii)} differs from Theorem \ref{a}. But a slightly weaker form of monotonicity of the posterior density $\pi(\cdot|Z^{(N)})$ still holds: the same arguments employed to prove part \textbf{ii)} of Theorem \ref{a} show that $\pi(\cdot|Z^{(N)})$ is decreasing on $\{\theta:\|\theta\|_{\R^D}\ge r N^{-b}\}$ (any $r>0$) along the half-lines through $0$, i.e.
	\begin{equation}\label{strano}
        \mathbb P_0^N\big( \pi(v e|Z^{(N)})\le \pi(v' e|Z^{(N)})~\textnormal{for all}~v\ge v'\ge r, e \in \R^D, \|e\|_{\R^D}=1\big) \xrightarrow{N\to\infty} 1.
	\end{equation}
	We note that this notion precludes the possibility of $\pi(\cdot|Z^{(N)})$ having extremal points outside of the region of dominant posterior mass, and implies that moving toward the origin will \textit{always} increase the posterior density. As a result, many typical Metropolis-Hastings would be encouraged to accept such `radially inward' moves, if they arise as a proposal. Thus, crucially, our exponential hitting time lower bound in part \textbf{iv)} arises not through multimodality, but merely through volumetric properties of high-dimensional Gaussian measures.
\end{remark}

\smallskip

\begin{remark}[On the step-size condition]\normalfont \label{steprem}
One may wonder whether larger step-sizes can help to overcome the negative result presented in the last theorem. If the step-sizes are `time-homogeneous' and $\gg N^{-b}$ on average, then we may hit the region where the posterior is supported at some time. This would happen `by chance' and not because the data (via $\ell_N$) would suggest to move there, and future proposals will likely be outside of that bulk region, so that the chain will either exit the relevant region again or become deterministic because an accept/reject step refuses to move into such directions. In this sense, a negative result for (polynomially) small step sizes gives fundamental limitations on the ability of the chain to explore the precise characteristics of the posterior distribution. We also remark that the Lipschitz-constants of $\nabla\ell(\theta)$ are of order $D$ or $D^{1+b}$ in the preceding theorems, respectively. A Markov chain obtained from discretising a continuous diffusion process (such as MALA discussed in the next subsection) will generally require step-sizes that are inversely proportional to that Lipschitz constant in order to inherit the dynamics from the continuous process. For such examples, Assumption \ref{ass:step-eta} is natural. But as discussed at the end of the introduction, there exists a variety of `non-local' MCMC algorithms for which this step size assumption may not be satisfied.

\end{remark}

\subsection{Implications for common MCMC methods with `cold-start'}\label{cst}

The preceding general hitting time bounds apply to commonly used MCMC methods in high-dimensional statistics. We focus in particular on algorithms that are popular with PDE models and inverse problems, see, e.g., \cite{CRSW13, BGLFS17} and also \cite{N22} for many more references. We illustrate this for two natural examples with Metropolis-Hastings adjusted random walk and gradient algorithms. Other examples can be generated without difficulty.

\subsubsection{Preconditioned Crank-Nicolson}
We first give some hardness results for the popular preconditioned Crank-Nicolson (pCN) algorithm. A dimension-free convergence analysis for pCN was given in the important paper by Hairer, Stuart, and Vollmer \cite{HSV14} based on ideas from \cite{HMS11}. The results in the present section show that while the mixing bounds from \cite{HSV14} are in principle uniform in $D$, the implicit dependence of the constants on the conditions on the log-likelihood-function in \cite{HSV14} can re-introduce exponential scaling when one wants to apply the results from \cite{HSV14} to concrete ($N$-dependent) posterior distributions. This confirms a conjecture about pCN made in Section 1.2.1 of \cite{NW20}.

Let $\mathcal C$ denote the covariance of some Gaussian prior on $\R^D$ with density $\pi$. Then the pCN algorithm for sampling from some posterior density $\pi(\theta |Z^{(N)})\propto e^{\ell_N(\theta)} \pi (\theta)$ is given as follows. Let $(\xi_k:k\ge 1)$ be an i.i.d.\ sequence of $\mcN(0,\mathcal C)$ random vectors. For initialiser $\vartheta_0\in \R^D$, step size $\beta>0$ and $k\ge 1$, the MCMC chain is then given by
\begin{enumerate}
	\item \textsc{Proposal:} $p_k \sim \sqrt{1-\beta}\vartheta_{k-1}+ \sqrt \beta \xi_k$,
	\item \textsc{Accept-reject:} Set
	\begin{equation}\label{pcn-algo}
	\vartheta_k= 
	\begin{cases}
		p_k ~~~\textnormal{w.p.}~~\min \big\{ 1,  e^{\ell_N(p_k)-\ell_N(\vartheta_{k-1})} \big\}, \\
		\vartheta_{k-1} ~~~ \textnormal{else}.
	\end{cases}
	\end{equation}
\end{enumerate}
By standard Markov chain arguments one verifies (see \cite{HSV14} or Ch.1~in \cite{N22}) that the (unique) invariant density of $(\vartheta_k:k\ge 1)$ equals $\pi(\cdot|Z^{(N)})$.

We now give a hitting time lower bound for the pCN algorithm which holds true in the regression setting for which the main Theorems \ref{a} and \ref{b} (for generic Markov chains) were derived. In particular, we emphasize that the lower bounds to follow hold for the choice of regression `forward' map $\mathcal G$ constructed in the proofs of Theorems \ref{a} and \ref{b}. As for the general results, we treat the two cases of $\mathcal C=I_D/D$ or $\mathcal C=\Sigma_\alpha$ separately.

\begin{theorem}\label{pcn}
	Let $\vartheta_k$ denote the pCN Markov chain from (\ref{pcn-algo}).
	
	\textbf{i)} Assume the setting of Theorem \ref{a} with $\mathcal C=I_D/D$, and let $\mathcal G$ be as in Theorem \ref{a}. Then there exist constants $c_1,c_2,\eps >0$ such that for any $\beta \le c_1$, there is an initialisation point $\vartheta_0\in \Theta_{1,\eps}$ such that the hitting time $\tau_{B_s} = \inf \{k: \vartheta_k\in B_s \}$ (for $B_s$ as in (\ref{Bs})) satisfies with high probability (under the law of the data and of the Markov chain) as $N\to\infty$ that $\tau_{B_s}\ge \exp \big(c_2 D \big)$.
	
	\smallskip 
	
	\textbf{ii)} Assume the setting of Theorem \ref{b} with $\mathcal C=\Sigma_\alpha$ for $\alpha>d/2$, and let $\mathcal G$ be as in Theorem \ref{a}. Then there exist constants $c_1,c_2,\eps>0$ such that if $\beta \le c_1 N^{-1-2b}$ there is an initialisation point $\vartheta_0\in \Theta_{N^{-b},\eps N^{-b}}$ such that the hitting time $\tau_{B_s} = \inf \{k: \vartheta_k\in B_s \}$ satisfies with high probability that $\tau_{B_s}\ge \exp \big(c_2 D \big)$.
	
\end{theorem}

\subsubsection{Gradient-based Langevin algorithms}

We now turn to \textit{gradient-based} Langevin algorithms which are based on the discretization of continuous-time diffusion processes \cite{CRSW13, D17}. A polynomial time convergence analysis for the \textit{unadjusted} Langevin algorithm in the strongly log-concave case has been given in \cite{D17, DM19} and also in \cite{rigetal} for the Metropolis-adjusted case (MALA). We show here that for unimodal but not globally log-concave distributions, the MCMC scheme can take an exponential time to reach the bulk of the posterior distribution. For simplicity we focus on the Metropolis-adjusted Langevin algorithm which is defined as follows. Let $(\xi_k:k\ge 1)$ be a sequence of i.i.d. $\mcN(0,I_D)$ variables, and let $\gamma>0$ be a step-size.

\begin{enumerate}
	\item \textsc{Proposal:} $p_{k} = \vartheta_{k-1} + \gamma \nabla \log \pi (\vartheta_{k-1}| Z^{(N)}) + \sqrt{2\gamma} \xi_{k}$.
	\item \textsc{Accept-reject:} Set
	\begin{equation}\label{mala}
		\vartheta_{k}= 
		\begin{cases}	p_{k}~~~\textnormal{w.p.}~~\min \Big\{ 1,  \frac{ \pi(p_{k}|Z^{(N)}) \exp\big(  -  \| \vartheta_{k-1} -p_{k}  -\gamma \nabla \log \pi(p_{k}|Z^{(N)} )\|^2 \big) }{\pi(\vartheta_{k-1}|Z^{(N)})  \exp\big(  -  \| p_{k} - \vartheta_{k-1} -\gamma \nabla \log \pi(\vartheta_{k-1}|Z^{(N)} )\|^2 \big) } \Big\}, \\
			\vartheta_{k-1} ~~~ \textnormal{else}.
		\end{cases}
	\end{equation}
\end{enumerate}

Again, standard Markov chain arguments show that $\Pi(\cdot|Z^{(N)})$ is indeed the (unique) invariant distribution of $(\vartheta_k:k\ge 1)$. We note here that for the forward $\mathcal G$ featuring in our results to follows, $\nabla \log \pi$ may only be well-defined (Lebesgue-) almost everywhere on $\R^D$ due to our piecewise smooth choice of $w$, see (\ref{was}) below. However, since all proposal densities involved possess a Lebesgue density, this specification almost everywhere suffices in order to propagate the Markov chain with probability~$1$. 
Alternatively one could also straightforwardly avoid this technicality by smoothing our choice of function $w$ in (\ref{was}), which we refrain from for notational ease.

\begin{theorem}\label{MALA} Let $\vartheta_k$ denote the MALA Markov chain from (\ref{mala}). 

\textbf{i)} Assume the setting of Theorem \ref{a}, with $\mcN(0, I_D/D)$ prior, and let $\mathcal G$ also be as in Theorem \ref{a}. There exists some $c_1,c_2,\eps>0$ such that if the step size of $(\vartheta_k:k\ge 1)$ satisfies $\gamma\le c_1/N$, then there is an initialisation point $\vartheta_0\in \Theta_{1,\eps}$ such that
the hitting time $\tau_{B_s} = \inf \{k: \vartheta_k\in B_s \}$ (for $B_s$ as in (\ref{Bs})) satisfies with high probability (under the law of the data and of the Markov chain) as $N\to\infty$ that $\tau_{B_s}\ge \exp \big(c_2 D \big)$.

\textbf{ii)} Assume the setting of Theorem \ref{b}, with a $\mcN(0, \Sigma_\alpha)$ prior, and let $\mathcal G$ also be as in Theorem \ref{b}. Then there exist some constant $c_1,c_2,\eps>0$ such that whenever $\gamma\le c_1N^{-1-b-2\alpha}$, there is an initialisation point $\vartheta_0\in \Theta_{N^{-b},\eps N^{-b}}$, such that the hitting time $\tau_{B_s} = \inf \{k: \vartheta_k\in B_s \}$ satisfies with high probability (under the law of the data and of the Markov chain) that $\tau_{B_s}\ge \exp \big(c_2 D \big)$.
\end{theorem}

As mentioned in Remark \ref{steprem}, a bound on the step-size that is inversely proportional to the Lipschitz constant of $\nabla \ell$ is natural for algorithms like MALA that arise from discretisation of a continuous time Markov process, see, e.g., \cite{DM19, rigetal}. We emphasise again that these Lipschitz constants are $D$- and $N$-dependent, so that the required bounds on $\gamma$ are not unnatural. `Optimal' step-size prescriptions for MALA \cite{RR01, BPS04, MPS12, rigetal} derived for Gaussian and log-concave targets or, more generally, mean field limits (in which the posterior distribution possesses a product or mean-field structure, unlike in the models considered here) would need to be adjusted to our model classes to be comparable.

\section{Proofs of the main theorems}\label{proofs}

We begin in Section \ref{radsec} by constructing the family of regression maps $\G$ underlying our results from Section \ref{sec:results_gaussian_priors}. Sections \ref{post-ratio} and \ref{hit-sec} reduce the hitting time bounds from Theorems \ref{a} and \ref{b} (for general Markov chains) to hitting time bounds for intermediate `free energy barriers' that the Markov chain needs to travel through. Subsequently, Theorems \ref{b} and \ref{a} are proved in Sections \ref{b-pf} and \ref{a-pf} respectively. Finally, the proofs for pCN (Theorem \ref{pcn}) and MALA (Theorem \ref{MALA}) are contained in Section \ref{algo-pfs}. 

\subsection{Radially symmetric choices of \texorpdfstring{$\mathscr G$}{G}}\label{radsec}

We start with our parameterisation of the map $\G$. In our regression model and since $\EE\varepsilon^2=1$,
\begin{equation}\label{liko}
\ell(\theta) =-\frac{N}{2}\mathrm E_{\theta_0}^1|Y-\mathscr G(\theta)(X)|^2 = -\frac{N}{2} \|\mathscr G(\theta_0) - \mathscr G(\theta)\|_{L^2}^2 - \frac{N}{2},~~~\theta \in \R^D.
\end{equation}
We have $\theta_0=0$ and by subtracting a fixed function $\G(0)$ from $\G(\theta)$ if necessary we can also assume that $\mathscr G(\theta_0)=0$. In this case, since $\mathrm{vol}(\mathcal X)=1$,
\begin{equation}\label{liko1}
\ell(\theta) =  -\frac{N}{2} \| \mathscr G(\theta)\|_{L^2}^2 - \frac{N}{2},
\end{equation}
Take a bounded continuous function $w: [0,\infty] \to [0, \|w\|_\infty)$ with a unique minimiser $w(0)=0$ and take $\mathscr G$ of the `radial' form $$\mathscr G(\theta) = \sqrt{w(\|\theta\|_{\R^D})} \times g(x),~ \theta \in \R^D, x \in \mathcal X,$$ where
$$g: \mathcal X \to [g_{\min}, g_{\max}],~0<g_{\min}<g_{\max}<\infty,~\|g\|_{L^2_\mu(\mathcal X)}=1.$$
The assumption $\mathcal{G}(\theta_0) = 0$ implies $Y_i=0+\varepsilon_i$ under $P_{\theta_0}^N$, so that we have
\begin{align}\label{likemp}
\ell_N(\theta) & = -\frac{1}{2}\sum_{i=1}^N|\varepsilon_i - \sqrt{w(\|\theta\|)}g(X_i)|^2  \notag \\
&= -\frac{w(\|\theta\|_{\R^D})}{2}\sum_{i=1}^N g^2(X_i) - \frac{1}{2} \sum_{i=1}^N \varepsilon_i^2 + \sqrt {w(\|\theta\|)}\sum_{i=1}^N \varepsilon_i g(X_i)
\end{align}
and the average log-likelihood is
\begin{equation} \label{likav}
\ell(\theta) = \mathrm E_{\theta_0}^N\ell_N(\theta) = -\frac{N}{2} w(\|\theta\|_{\R^D}) - \frac{N}{2}, \theta \in \R^D.
\end{equation}
Define $\epsilon$-annuli of Euclidean space 
\begin{equation}\label{thetas}
\Theta_{r,\epsilon} =  \big\{\theta \in \R^D: \|\theta\|_{\R^D} \in (r, r+\epsilon)\big\},~r \ge0.
\end{equation} 
We then also set, for any $s\ge 0,~\epsilon>0$, $$w_-(r, \epsilon) = \inf_{s \in(r, r+\epsilon)}w(s),~~w_+(r, \epsilon)=\sup_{s \in (r, r+\epsilon)}w(s).$$ For our main theorems the map $w$ will be monotone increasing and the preceding notation $w_-, w_+$ is then not necessary, but Proposition \ref{nubd} is potentially also useful in non-monotone settings (as remarked after its proof), hence the slightly more general notation here. 

\smallskip

The choice that $\G$ is radial is convenient in the proofs, but means that the model is only identifiable up to a rotation for $\theta \neq 0$. One could easily make it identifiable by more intricate choices of $\G$, but the main point for our negative results is that the function $\ell$ has a unique mode at the ground truth parameter $\theta_0$ and is identifiable there. 

\subsubsection{A locally log-concave, globally monotone choice of \texorpdfstring{$w$}{w}}

Define for $t<L$ and any $r>0$ the function $w: [0,\infty) \to \R$ as

\medskip
\begin{minipage}{0.45\textwidth}
\begin{align}\label{was}
w(r) &= 4(Tr)^21_{[0,t/2)} \\  
&~~~~+ [(Tt)^2 + T (r-t/2)] 1_{[t/2,t)}(r) \notag \\
&~~~~+ [(Tt)^2+(Tt/2) + \rho (r-t)]1_{[t,L)}(r) \notag \\
&~~~~+ [(Tt)^2+(Tt/2)+\rho(L-t)]1_{[L,\infty)}(r), \notag
\end{align}
\hspace{0.1\linewidth}
\end{minipage}
\begin{minipage}{0.35\textwidth}
\centering
  \includegraphics[width=\textwidth]{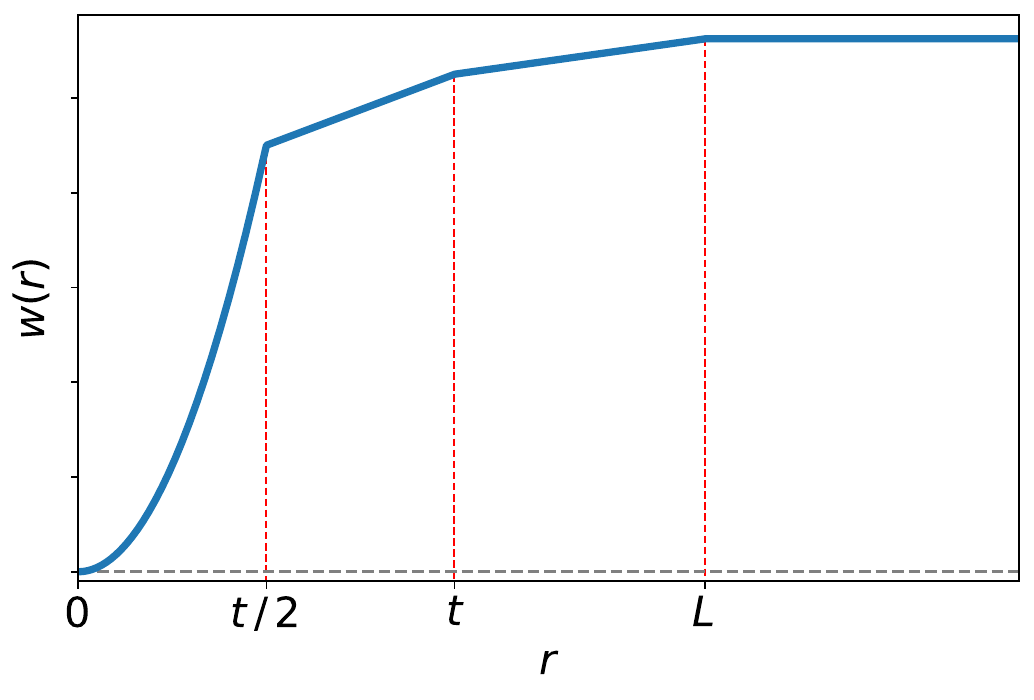}
\end{minipage}

\noindent
where $T> \rho,$ are fixed constants to be chosen. Note that $w$ is monotone increasing and 
\begin{equation}
\|w\|_\infty = (Tt)^2+ (Tt/2) + \rho(L-t)<\infty.
\end{equation}
The function $w$ is quadratic near its minimum at the origin up until $t/2$, from when onwards it is piece-wise linear. In the linear regime it initially has a `steep' ascent of gradient $T$ until $t$, then grows more slowly with small gradient $\rho$ from $t$ until $L$, and from then on is constant.
The function $w$ is not $C^\infty$ at the points $r=t/2, r=t, r=L$, but we can easily make it smooth by convolving with a smooth function supported in small neighbourhoods of its breakpoints $r$ without changing the findings that follow. We abstain from this to simplify notation.

The following proposition summarises some monotonicity properties of the empirical log-likelihood function arising from the above choice of $w$.

\begin{proposition}\label{mono}
Let $w$ be as in (\ref{was}). Then there exists $C>0$ such that for any $r_0>0$ and $N\ge 1$, we have
\[ P_0^N\big( \sup_{r_0\le r<s\le L}~\sup_{\|\theta_s\|=s,\|\theta_r\|=r} \frac{\ell_N(\theta_s) - \ell_N(\theta_r)}{w(s)-w(r)} \le - \frac N4\big) \ge 1 -\frac CN - \frac C{N w(r_0)}. \]
In particular if $r_0<t/2$ is such that $(Tr_0)^2N \to \infty$ as $N\to \infty$ then the r.h.s.~is $1-o(1)$.
\end{proposition}

\begin{proof}
Recalling (\ref{likav}), (\ref{likemp}) and since $w$ is monotonically increasing, we bound
\begin{align*}
&P_0^N(\ell_N(\theta_s) - \ell_N(\theta_r) > (N/4) (w(r)-w(s))) \\
&=P_0^N\big(\ell_N(\theta_s)-\ell_N(\theta_r) - (\ell(\theta_s)-\ell(\theta_r)) > -\frac{N}{4}(w(r)-w(s))\big) = \\
& \Pr\Big(\frac{(w(r)-w(s))}{2}\sum_{i=1}^N (g^2(X_i)-1) + \big(\sqrt{w(s)}-\sqrt{w(r)}\big) \sum_{i=1}^N\varepsilon_i g(X_i) > \frac{N}{4}(w(s)-w(r)) \Big) \\
&= \Pr\Big(-\sum_{i=1}^N (g^2(X_i)-\EE g^2(X))/2 + \frac{1}{\sqrt {w(s)} + \sqrt {w(r)}}\sum_{i=1}^N\varepsilon_i g(X_i) > \frac{N}{4} \Big) \\
&\le \Pr\Big(\big|\sum_{i=1}^N (g^2(X_i)-\EE g^2(X)) \big|> N/4 \Big) + \Pr\Big(|\sum_{i=1}^N\varepsilon_i g(X_i)| > \frac{2N \sqrt {w(r_0)}}{8} \Big) \\
&= \mcO(1/N) + \mcO(1/(N w(r_0))),
\end{align*}
using Chebyshev's inequality in the last step. Since the events in the penultimate step do not depend on $r < s\in [r_0,L]$, the result follows.

\end{proof}

\subsection{Bounds for posterior ratios of annuli}\label{post-ratio}

A key quantity in the proofs to follow will be to obtain asymptotic $(N \to \infty)$ bounds of the following functional (recalling the definition of the Euclidean annuli $\Theta_{r,\eps}$ from (\ref{thetas})),
\begin{equation}\label{fp}
\mathscr F_N(r,\eps) =  \frac{1}{N}  \log \int_{\Theta_{r,\eps}} e^{\ell_N(\theta)}d\Pi(\theta),~~ r\ge 0,~\eps >0,
\end{equation}
in terms of the map $w$. 
As a side note, we remark that this functional has a long history in the statistical physics of glasses, in which it is often referred to as the \emph{Franz-Parisi} potential \cite{franz1995recipes,bandeira2022franz}.

\begin{proposition}\label{nubd}
Consider the regression model (\ref{model}) with radially symmetric choice of $\G$ from Subsection \ref{radsec} such that $\|w\|_\infty \le W$ for some fixed $W<\infty$ (independent of $D,N$), and let $\Pi = \Pi_N$ denote a sequence of prior probability measures on $\R^D$.

\textbf{i)} Suppose that for some radii $0<s<\sigma$, constants $\eps,\eta,\nu>0$ and for all $N\ge 1$ large enough, we have
\begin{equation}\label{hypo}
\frac{1}{N}\log \frac{\Pi(\Theta_{s,\eta})}{\Pi(\Theta_{\sigma,\eps})} \le -2\nu - \frac{(w_+(\sigma,\eps) -w_-(s,\eta))}{2}.
\end{equation}
Then the posterior distribution $\Pi(\cdot|Z^{(N)})$ from (\ref{post}) arising in the model (\ref{model}) satisfies that with high $P_{0}^N$-probability as $N\to \infty$,
\begin{equation}\label{bdv}
\frac{\Pi(\Theta_{s,\eta}|Z^{(N)})}{\Pi(\Theta_{\sigma,\eps}|Z^{(N)})} \le e^{-\nu N}.
\end{equation}

\textbf{ii)} If in addition $w$ is monotone increasing on $[0,\infty)$ and if for some $L>1+\eps$,
	\begin{equation}\label{L-hypo}
		\frac{1}{N}\log \frac{\Pi(B_{L}^c)}{\Pi(\Theta_{\sigma,\eps})} \le -2\nu,
\end{equation}
then the posterior distribution $\Pi(\cdot|Z^{(N)})$ also satisfies (with high probability as $N\to\infty$) that
\begin{equation}\label{L-ratio}
	\frac{\Pi(B_{L}^c|Z^{(N)})}{\Pi(\Theta_{\sigma,\eps}|Z^{(N)})} \le e^{-\nu N}.
\end{equation}
\end{proposition}

\begin{remark}[The prior condition for \texorpdfstring{$w$}{w} from (\ref{was})]
    \normalfont
If $\sigma>s>t$, for $w$ from (\ref{was}), the `likelihood' term in Proposition \ref{nubd} is
\begin{equation}\label{balance}
 \frac{w_+(\sigma,\eps)}{2} - \frac{w_-(s,\eta)}{2} = \frac{\rho(\sigma+\eps -t) -\rho(s-t) }{2} = \frac{\rho}{2}(\sigma+\eps -s)>0,
\end{equation}
so that if we also assume
\begin{equation}\label{supbd}
Tt + \rho L = \smallO(\sqrt N)
\end{equation}
to control $\omega_N, \omega_N'$ in the proof that follows, then to verify (\ref{hypo}) it suffices to check 
\begin{equation}\label{bal2}
\frac{1}{N}\log \frac{\Pi(\Theta_{s,\eta})}{\Pi(\Theta_{\sigma,\eps})} \le -2\nu - \frac{\rho}{2}(\sigma+\eps -s) \end{equation}
for all large enough $N$.
\end{remark}

\begin{proof}
	\textbf{Proof of part i).}
From the definition of $\ell_N$ in (\ref{likemp}) we first notice that for all $r\ge 0,\epsilon>0$,
\begin{equation*}
\inf_{\theta \in \Theta_{r,\epsilon}} \ell_N(\theta) \ge -\frac{1}{2}\sum_{i=1}^N\varepsilon_i^2 - \frac{1}{2}w_+(r,\epsilon) \sum_{i=1}^N g^2(X_i) - \sqrt{w_+(r,\epsilon)} |\sum_{i=1}^N \varepsilon_i g(X_i)|,
\end{equation*}
and
\begin{equation*}
\sup_{\theta \in \Theta_{r,\epsilon}} \ell_N(\theta) \le -\frac{1}{2}\sum_{i=1}^N\varepsilon_i^2 - \frac{1}{2}w_-(r,\epsilon) \sum_{i=1}^N g^2(X_i) + \sqrt{w_+(r,\epsilon)} |\sum_{i=1}^N \varepsilon_i g(X_i)|,
\end{equation*}

\smallskip

We can now further bound, for our $\mathscr G$,

\begin{align*}
 \frac{1}{N} \log \int_{\Theta_{r,\epsilon}} e^{\ell_N(\theta)} d\Pi(\theta) &\le -\frac{1}{2N}\sum_{i=1}^N\varepsilon_i^2 - \frac{w_-(r,\epsilon) }{2N}\sum_{i=1}^N g^2(X_i) + \frac{\sqrt{w_+(r,\epsilon)}}{N} \big|\sum_{i=1}^N \varepsilon_i g(X_i)\big| + \frac{\log \Pi(\Theta_{r,\epsilon})}{N} 
\end{align*}
and 
\begin{align*}
 \frac{1}{N} \log \int_{\Theta_{r,\epsilon}} e^{\ell_N(\theta)} d\Pi(\theta) \ge-\frac{1}{2N}\sum_{i=1}^N\varepsilon_i^2 - \frac{w_+(r,\epsilon) }{2N}\sum_{i=1}^N g^2(X_i) - \frac{\sqrt{w_+(r,\epsilon)}}{N} \big|\sum_{i=1}^N \varepsilon_i g(X_i)\big|+ \frac{\log \Pi(\Theta_{r,\epsilon})}{N} 
\end{align*}
We estimate $\sqrt {w_+(r,\epsilon)} \le \bar w(r,\epsilon) = \max(w_+(r,\epsilon), 1)$, and noting that $$\EE\varepsilon_i^2 = 1 = \EE g^2(X_i), ~~\EE\varepsilon_i g(X_i)=0$$ we can use Chebyshev's (or Bernstein's) inequality to construct an event of high probability such that the functional $\mathscr F_N$ from (\ref{fp}) is bounded as
\begin{align}\label{fup}
\mathscr F_N(r,\epsilon) &\le -\frac{1}{2} - \frac{w_-(r,\epsilon) }{2}+ \frac{\log \Pi(\Theta_{r,\epsilon})}{N} +\omega_N(s,\eta)
\end{align}
and 
\begin{align} \label{flow}
\mathscr F_N(r,\epsilon) & \ge-\frac{1}{2} - \frac{w_+(r,\epsilon) }{2} + \frac{\log \Pi(\Theta_{r,\epsilon})}{N} +\omega'_N(r,\epsilon),
\end{align}
where
\begin{equation}\label{omegan}
\omega_N(r,\epsilon) = \mcO\Big(1+\frac{w_-(r,\epsilon) + \bar w(r,\epsilon)}{\sqrt N}\Big), ~\omega_N'(s) = \mcO\Big(1+\frac{w_+(r,\epsilon) + \bar w(r,\epsilon)}{\sqrt N}\Big),
\end{equation}
and this is uniform in all $(r,\epsilon)$ since $\|w\|_\infty \le W$ is bounded. Using the above with $(r,\epsilon)$ chosen as $(s,\eta)$ and $(\sigma,\eps)$ respectively, we then obtain
\begin{align}\label{intest}
\frac{1}{N}\log \frac{\Pi(\Theta_{s,\eta}|Z^{(N)})}{\Pi(\Theta_{\sigma,\eps}|Z^{(N)})}  &= \mathscr F_N(s,\eta) - \mathscr F_N(\sigma,\eps) \notag \\
&\le  - \frac{w_-(s,\eta)}{2} + \frac{\log \Pi(\Theta_{s,\eta})}{N}   + \frac{w_+(\sigma,\eps)}{2} -  \frac{\log \Pi(\Theta_{\sigma,\eps})}{N} +\omega_N(s,\eta) - \omega_N'(\sigma,\eps) \notag \\
& = - \frac{w_-(s,\eta)}{2} + \frac{w_+(\sigma,\eps)}{2} + \frac{1}{N}\log \frac{\Pi(\Theta_{s,\eta})}{\Pi(\Theta_{\sigma,\eps})} +\omega_N(s,\eta) - \omega_N'(\sigma,\eps)
\end{align}
with high $P_{\theta_0}^N$-probability. The result now follows from the hypothesis (\ref{hypo}) and since the terms $\omega_N, \omega_N'$ are $\smallO(1)$.

\smallskip

\textbf{Proof of part ii).} The proof of part \textbf{ii)} follows from an obvious modification of the previous arguments.

\end{proof}

In the case where $\Pi(\Theta_{s,\eta})$ and $\Pi(\Theta_{\sigma,\eps})$ are comparable (so that the l.h.s.~in (\ref{hypo}) converges to zero), a local optimum at $\sigma$ in the function $w$ away from zero can verify the last inequality for `intermediate' $s$ such that $w(s)-w(\sigma)\le-2\nu$. This can be used to give computational hardness results for MCMC of multi-modal distributions. But we are interested in the more challenging case of `unimodal' examples $w$ from (\ref{was}). Before we turn to this, let us point out what can be said about the hitting times of Markov chains if the conclusion (\ref{bdv}) of Proposition~\ref{nubd} holds.

\subsection{Bounds for Markov chain hitting times}\label{hit-sec}

\subsubsection{Hitting time bounds for intermediate sets \texorpdfstring{$\Theta_{s,\eta}$}{}}

In \eqref{bdv} we can think of $\Theta_{\sigma,\eps}$ as the `initialisation region' (further away from $\theta_0$) and $\Theta_{s,\eta}$ for intermediate $s$ is the `barrier' before we get close to $\theta_0=0$.
The last bound permits the following classic hitting time argument, taken from \cite{arous2020free}, see also \cite{J03}.

\begin{proposition}\label{hit}
Consider any Markov chain $(\vartheta_k:k \in \mathbb N)$ with invariant measure $\mu=\Pi(\cdot|Z^{(N)})$ for which (\ref{bdv}) holds. For constants $\eta<\sigma-s$, suppose $\vartheta_0$ is started in $\Theta_{\sigma,\eps}$, $\mu(\Theta_{\sigma,\eps})>0$, drawn from the conditional distribution $\mu(\cdot|\Theta_{\sigma,\eps})$, and denote by $\tau_s$ the hitting time of the Markov chain onto $\Theta_{s,\eta}$, that is, the number $\tau_s$ of iterates required until $\vartheta_{k}$ visits the set $\Theta_{s,\eta}$. 
Then $$\Pr(\tau_s \le K) \le K e^{-\nu N},~~K>0.$$ 
Similarly, on the event where (\ref{L-ratio}) holds we have that 
\[ \Pr(\tau_{B_{L}^c} \le K) \le Ke^{-\nu N},~~K>0. \]
\end{proposition}

\begin{proof}[Proof of Proposition~\ref{hit} --] We have
\begin{align*}
\Pr(\tau_s \le K) &= \Pr (\vartheta_k \in \Theta_{s,\eta} \text{ for some } 1 \le k \le K|\vartheta_0 \in \Theta_{\sigma,\eps}) \\
&= \frac{\Pr(\vartheta_0 \in \Theta_{\sigma,\eps}, \vartheta_k \in \Theta_{s,\eta} \text{ for some } 1 \le k \le K)}{\mu(\Theta_{\sigma,\eps})} \le \frac{\sum_{k \le K} \Pr(\vartheta_k \in \Theta_{s,\eta})}{\mu(\Theta_{\sigma,\eps})} \\
&\le K \frac{\mu(\Theta_{s,\eta})}{\mu(\Theta_{\sigma,\eps})} \le K e^{-\nu N}.
\end{align*}
The second claim is proved analogously.
\end{proof}

The last proposition holds `on average' for initialisers $\vartheta_0 \sim \mu(\cdot|\Theta_{\sigma,\eps})$, and since $\Pr = \EE_{\mu(\cdot|\Theta_{\sigma,\eps})} \Pr_{\vartheta_0}$ where $\Pr_{\vartheta_0}$ is the law of the Markov chain started at $\vartheta_0$, the hitting time inequality holds at least for one point in $\Theta_{\sigma,\eps}$ since $\inf_{\vartheta_0} \Pr_{\vartheta_0} \le \EE_{\mu(\cdot|\Theta_{\sigma,\eps})} \Pr_{\vartheta_0}$.

\subsubsection{Reducing hitting times for \texorpdfstring{$B_{s}$}{} to ones for \texorpdfstring{$\Theta_{s,\eta}$}{}}\label{reduction}

We now reduce part \textbf{iv)} of Theorems~\ref{a} and \ref{b},  i.e.\ bounds on the hitting time of the region $B_s$ in which the posterior contracts, to a bound for the hitting time $\tau_s$ for the annulus $\Theta_{s,\eta}$, which is controlled in Proposition \ref{hit}. To this end, in the case of Theorem \ref{a}, we suppose that Propositions \ref{nubd} and \ref{hit} are verified with $\nu=\sigma=1$, some $\eps>0$ and $L,s,\eta$ as in the theorem, and in the case of Theorem \ref{b}, we assume the same with choice $\sigma=N^{-b}$ and $\nu>0$ given after (\ref{bal2a}) below. 
For $c_0$ from Assumption \ref{ass:step-eta}, define the events
\[ A_N := \{\forall k\le e^{(\nu \wedge c_0)N/2}: \|\vartheta_{k+1}- \vartheta_k \|_{\R^D} \le \eta/2 \}. \]
We can then estimate, using Assumption \ref{ass:step-eta}, that on the frequentist event on which Proposition \ref{hit} holds (which we apply with $K= e^{(\nu \wedge c_0)N/2}\le e^{\nu N/2}$), under the probability law of the Markov chain we have
\begin{align*}
	\Pr (\tau_{B_s}\le e^{(\nu \wedge c_0 )N/2} )&\le \Pr(\tau_{B_s}\le e^{(\nu \wedge c_0)N/2},~A_N) +\Pr(A_N^c )\\
	&\le \Pr(\tau_{s}\le e^{(\nu \wedge c_0)N/2}) +\Pr(A_N^c,~\tau_{B^c_{L}} >e^{(\nu \wedge c_0) N/2} )+ \Pr(\tau_{B^c_{L}} \le e^{(\nu\wedge c_0 )N/2} )\\
	&\le 2e^{-\nu N/2} +  e^{(\nu \wedge c_0)N/2} \sup_{\theta\in B_{L}} \mathcal P_N(\theta,\{\vartheta: \|\theta-\vartheta\|_{\R^D} \ge \eta/2 \})\\
	&\le  2e^{-(\nu \wedge c_0) N/2} + e^{(\nu \wedge c_0)N/2- c_0 N} \le 3e^{-(\nu\wedge c_0)N/2},
\end{align*}
where in the second inequality we have used that on the events $A_N$, the Markov chain $\vartheta_k$, when started in $\Theta_{1,\eps}$, needs to pass through $\Theta_{s,\eta}$ in order to reach $B_s$.

\subsection{Proof of Theorem \ref{b}}\label{b-pf}

In this section, we use the results derived in the previous part of Section~\ref{proofs} to finish the proof of Theorem~\ref{b}. 
Parts \textbf{i)} and \textbf{ii)} of the theorem follow from Proposition \ref{mono} and our choice of $w$ in (\ref{was}). 
We therefore concentrate on the proofs of part \textbf{iii)} and \textbf{iv)}. We start with proving a key lemma on small-ball estimates for truncated $\alpha$-regular Gaussian priors.

\subsubsection{Small ball estimates for \texorpdfstring{$\alpha$}{alpha}-regular priors}\label{subsubsec:def_matern_prior}

Let us first define precisely the notion of $\alpha$-regular Gaussian priors. For some fixed $\alpha > d/2$, the prior $\Pi$ arises as the truncated law $Law(\theta)$ of an $\alpha$-regular Gaussian process with RKHS $\mathcal H = H^\alpha$, a Sobolev space over some bounded domain/manifold $\mathcal X$, see e.g., Sec.~6.2.1~in \cite{N22} for details. Equivalently (under the Parseval isometry) we take a Gaussian Borel measure on the usual sequence space $\ell_2 \simeq L^2$ with RKHS equal to $$h^\alpha=\Big\{(\theta_i)_{i=1}^\infty: \sum_{i=1}^\infty i^{2(\alpha/d)} \theta_i^2 = \|\theta\|_{H^\alpha}^2 <\infty \Big\}, ~\alpha>d/2.$$ The prior $\Pi$ is the truncated law of $\theta_D=(\theta_1, \dots, \theta_D), D \in \mathbb N$.

\begin{lemma}\label{asb}
Fix $z>0,~\alpha>d/2$ and $\kappa>0$, and set
$$b = \frac{\alpha}{d} - \frac{1}{2}, ~\tau=\frac{1}{b}=\frac{2d}{2\alpha-d}.$$
Then if $D/N \simeq \kappa>0$, there exist constants $\bar c_0 > c_0$ (depending on $b,\kappa$) such that for all $N$ ($\ge N_0(z,b)$) large enough:
\begin{equation}\label{sbD}
c_0 (z+\kappa^{-\alpha/d}z^{-\tau/2})^{-\tau} \le  -\frac{1}{N} \log \Pi(\|\theta\|_{\R^D} \le zN^{-b}) \le \bar c_0  z^{-\tau}.
\end{equation}
\end{lemma}
\begin{proof}[Proof of Lemma~\ref{asb} --]
Note first that the $L^2$-covering numbers of the ball $h(\alpha, B)$ of radius $B$ in $H^\alpha$ satisfy the well-known two-sided estimate
\begin{equation}\label{mett}
\log \mcN(\delta, \|\cdot\|_{L^2}, h(\alpha, B)) \simeq \Big(\frac{AB}{\delta} \Big)^{d/\alpha}, ~0<\delta<AB
\end{equation}
for equivalence constants in $\simeq$ depending only on $d, \alpha$. The upper bound is given in Proposition 6.1.1 in \cite{N22} and a lower bound can be found as well in the literature \cite{ET96} (by injecting $H^\alpha(\mathcal X_0)$ into $\tilde H^\alpha(\mathcal X)$ for some strict sub-domain $\mathcal X_0 \subset \mathcal X$, and using metric entropy lower bounds for the injection $H^\alpha(\mathcal X_0) \xhookrightarrow{} L^2(\mathcal X_0)$). 

Using the results about small deviation asymptotics for Gaussian measures in Banach space \cite{LL99} -- specifically Theorem 6.2.1 in \cite{N22} with $a=\frac{2d}{2\alpha-d}$ -- and assuming $\alpha>d/2$, this means that the concentration function of the 'untruncated prior' satisfies the two-sided estimate
\begin{equation}
-\log \Pi(\|\theta\|_{L^2} \le \gamma) \simeq \gamma^{-\frac{2d}{2\alpha-d}} = \gamma^{-\tau}, ~~ \gamma \to 0.
\end{equation}
Here, restricting to $\gamma \in (0,1)$, the two-sided equivalence constants depend only on $\alpha, d$. Setting 
\begin{equation}\label{bee}
\gamma = zN^{-b},~z>0,
\end{equation}
and noting that $b \tau=1$, we hence obtain that for some constants $ c_l,c_u>0$,
\begin{equation} \label{sb}
e^{-c_l z^{-\tau} N} \le \Pi(\|\theta\|_{L^2} \le zN^{-b}) \le e^{-c_u z^{-\tau} N},~~~ \text{any } z>0.
\end{equation}
We now show that as long as $D/N \approx \kappa>0$, one may use the above asymptotics to derive the desired small ball probabilities for the projected prior on $\R^D$.

We obviously have, by set inclusion and projection,
$$\Pi(\|\theta\|_{\R^D} \le zN^{-b}) \ge \Pi(\|\theta\|_{L^2} \le zN^{-b})$$ and hence it only remains to show the first inequality in eq.~\eqref{sbD}.
The Gaussian isoperimetric theorem (Theorem 2.6.12 in \cite{GN16}) and (\ref{sb}) imply that for $m\ge 4\sqrt {c_l}$ and some $c>0$, we have that (with $\Phi$ denoting the c.d.f. for $\mcN(0,1)$)
\begin{align*}
  \Pi\big(\theta = \theta_1 + \theta_2, \|\theta_1\|_{L^2}& \le z N^{-b}, \|\theta_2\|_{h^\alpha} \le m z^{-
\tau/2} \sqrt N \big) \\
&\ge \Phi\big(\Phi^{-1}\big(\Pi(\{\theta:\|\theta\|\le zN^{-b}\})\big)+mz^{-\tau/2}\sqrt N \big)\\
&\ge \Phi\big(-\sqrt{2c_l}z^{-\tau/2}\sqrt N+mz^{-\tau/2}\sqrt N \big)
\ge 1 - e^{-cz^{-\tau}N}
\end{align*}
(see also the proof of Lemma 5.17 in \cite{MNP21b} for a similar calculation). Then if the event in the last probability is denoted by $I$ we have 
\[\Pi(\|\theta_D\|_{\R^D} \le z N^{-b}) \le \Pi(\|\theta_D\|_{\R^D} \le z N^{-b}, I) +  e^{-cz^{-\tau}N}.\]
On $I$, if $D/N \to \kappa>0$ and by the usual tail estimate for vectors in $h^\alpha$, we have for some $c'>0$ the bound 
\[\|\theta-\theta_D\|_{L^2} \le \|\theta_1\|_{L^2} + c'D^{-\alpha/d} z^{-\tau/2}\sqrt N \leq zN^{-b}+ c'\kappa^{-\alpha/d}z^{-\tau/2} N^{-b}\] so that for any $z>0$,
\begin{align*}
\Pi(\|\theta_D\|_{\R^D} \le z N^{-b}) &\le \Pi(\|\theta\|_{L^2} \le zN^{-b} + \|\theta-\theta_D\|_{L^2}, I)+ e^{-cz^{-\tau}N} \\
&\le \Pi(\|\theta\|_{L^2} \le (2z+c'\kappa^{-\alpha/d}z^{-\tau/2}) N^{-b}) + e^{-cz^{-\tau}N} \\
&\le e^{-c_u (2z+c'\kappa^{-\alpha/d}z^{-\tau/2})^{-\tau}N}+ e^{-cz^{-\tau}N},
\end{align*}
and hence the lemma follows by appropriately choosing $c_0>0$.
\end{proof}

\smallskip

\begin{remark}\label{rescale} \normalfont For statistical consistency proofs in non-linear inverse problems, often \textit{rescaled} Gaussian priors are used to provide additional regularisation \cite{MNP21b, NW20, BN21}. For these priors a computation analogous to the previous lemma is valid: specifically if we rescale $\theta$ by $\sqrt N \delta_N$, where $\delta_N = N^{-\alpha/(2\alpha+d)}$ so that $\sqrt N \delta_N = N^{(d/2)/(2\alpha+d)} = N^{k}$, then we just take $N^{-\beta + k} = N^{-b}$ in the above small ball computation, that is $-b = -\beta+k$ or $b=\beta-k$, and the same bounds (as well as the proof to follow) apply.
\end{remark}

\subsubsection{Proof of Theorem \ref{b}, part \textbf{iv)}}

Lemma \ref{asb} and the hypotheses on $\eta$ immediately imply
$$ \Pi(\theta\in \Theta_{s,\eta})=\Pi\big(\|\theta\|_{\R^D} \in (s_bN^{-b}, s_bN^{-b}+\eta)\big) \le \Pi\big(\|\theta\|_{\R^D} \le 2s_bN^{-b}\big) \le e^{-c_0 N (2s_b+ \kappa^{-\alpha/d}(2s_b)^{-\tau/2})^{-\tau}}.$$
To lower bound $\Pi(\Theta_{N^{-b},(1+\eps)N^{-b}})$, we choose $\eps$ large enough such that
\[ \bar c_0(1+\eps)^{-\tau} < c_0(1+\kappa^{-\alpha/d})^{-\tau}, \]
which implies for all $N$ large enough that
\begin{align}\label{ann-lb}
\Pi\big(\|\theta\|_{\R^D} \in (N^{-b}, (1+\eps)N^{-b})\big) &= \Pi(\|\theta\|_{\R^D} \le (1+\eps)N^{-b}) - \Pi(\|\theta\|_{\R^D} \le N^{-b})) \notag \\
&\ge e^{-\bar c_0 (1+\eps)^{-\tau}N} - e^{-c_0(1+\kappa^{-\alpha/d})^{-\tau} N}\\
&\ge e^{-2\bar c_0 (1+\eps)^{-\tau} N}.\notag
\end{align}
Now, for $w$ from (\ref{was}), we choose
\begin{equation}\label{choices}
t=t_b N^{-b}, L=L_b N^{-b}, \rho \in (0,1], ~~  0<t_b<s_b<1/2<L_b<\infty,~~T = T_b N^{b},
\end{equation}
for $T_b, \rho,s_b,t_b$ fixed constants to be chosen, so that $\|w\|_\infty$ is bounded (uniformly in $N$) by a constant which depends only on $T_b, L_b, \rho$, whence (\ref{supbd}) holds. Now the key inequality (\ref{bal2}) with $s=s_bN^{-b}$ and with our choice of $\eta, \eps$, $\sigma = N^{-b}$ will be satisfied if
\begin{equation} \label{bal2a}
c_0 (2s_b+ \kappa^{-\alpha/d}(2s_b)^{-\tau/2})^{-\tau} \ge 2\bar c_0(1+\eps)^{-\tau} + 2\nu + \frac{\rho}{2}N^{-b}( 1+\eps -s_b).
\end{equation}
We define $\nu$ to equal to $1/3$ of the l.h.s. so that (\ref{bal2a}) will follow for the given $s_b,\kappa, \alpha, d$ by choosing $\eps$ large enough and whenever $N$ is large enough.

Finally, let us notice that with $L=C\sqrt N$ for some $C\ge 2 \EE[\|\theta\|_{\ell_2}]$, where $\theta$ is the infinite Gaussian vector with RKHS $h^\alpha$, we can deduce from Theorem 2.1.20 and Exercise 2.1.5 in \cite{GN16} that
\[ \Pr(\|\theta\|_{\R^D}\ge L)\le 2\exp(-c \, C^2N/2),~\text{ some }~c>0. \]
Thus, using also (\ref{ann-lb}), choosing $C$ large enough verifies (\ref{L-hypo}). Since (\ref{ann-lb}) and the a.s.~boundedness of $\sup_{\theta}|\ell_N(\theta)|$ for $\ell_N$ from (\ref{likemp}) imply that $\Pi(\Theta_{N^{-b}, (1+\varepsilon) N^{-b}}|Z^{(N)})>0$ a.s., Proposition \ref{nubd} and then also Proposition \ref{hit} apply for this prior, and the arguments from Section \ref{reduction} yield the desired result.

\subsubsection{Proof of Theorem \ref{b}, part \textbf{iii)}}

We finish the proof of the theorem by showing point $\textbf{iii)}$. We use the setting and choices from the previous subsection. Let us write $\mathbb G(A) = \int_A e^{\ell_N(\theta)}d\Pi(\theta)$ for any measurable set $A$. Recall the notation $B_{r} = \{\theta: \|\theta\|_{\R^D} \le r\}, r>0$. Repeating the argument leading to (\ref{flow}) with $B_{t/2}$ in place of $\Theta_{r,\epsilon}$, and using Lemma \ref{asb}, we have with high probability
$$\frac{1}{N} \log \mathbb G(B_{t/2}) \ge  -\frac{1}{2} - \frac{\sup_{r \le t_b N^{-b}/2} w(r)}{2} - \bar c_0\big(\frac{t_b}{2}\big)^{-\tau} + \omega_N'(t/2),$$ where $\omega_N'(t/2)=\mcO(\|w\|_\infty/\sqrt N)=\smallO(1)$ . Likewise, we also have
$$\frac{1}{N} \log \mathbb G(B_s^c) \le - \frac{1}{2} - \frac{\inf_{r \ge s_bN^{-b}} w(r)}{2} + \frac{1}{N}\log \Pi(B^c_s) +\omega''_N(s),$$ where $\omega''_N(s)=\mcO(\|w\|_\infty/\sqrt N)= \smallO(1)$. We can assume that $\mathbb G(B_s^c)>0$. Hence, since $\Pi(B^c_s)\to 1$ in view of Lemma \ref{asb},
\begin{align}
\frac{1}{N} \log \frac{\mathbb G(B_{t/2})}{\mathbb G (B_s^c)} &\ge -\frac{(Tt)^2}{2} - c_0\big(\frac{t_b}{2}\big)^{-\tau} + \frac{(Tt)^2 +(Tt/2) + \rho(s-t)}{2} + \frac{1}{N}\log \Pi(B^c_s) +\smallO(1) \notag \\
& \ge  \frac{T_bt_b}{4} - c_0\big(\frac{t_b}{2}\big)^{-\tau} +  \smallO(1). \label{t4}
\end{align}
Now, for $t_b<s_b$ fixed we can choose $T_b$ large enough such that the last quantity exceeds $1$ with high probability (in particular this retrospectively justifies the last $\smallO(1)$ as then $\|w\|_\infty=\mcO(1)$ for our choice of $T_b$). Therefore, again with high probability
\begin{equation}\label{ratioag}
\frac{\mathbb G(B_{t/2})}{\mathbb G (B_s^c)} \ge e^{N}\times (1+\smallO(1)).
\end{equation}
For $M_{t,s}=\{\theta: t/2<\|\theta\|_{\R^D} \le s \}$ this further implies that with high probability
$$\frac{\mathbb G(B_{t/2}) + \mathbb G(M_{t,s})}{\mathbb G(B_s^c)} \ge e^{ N} \times (1+\smallO(1)), $$
and then,
\begin{align*}
\Pi(B_{s}|Z^{(N)}) & = \frac{\mathbb G(B_{t/2}) + \mathbb G(M_{t,s})}{\mathbb G(B_{t/2})+ \mathbb G(M_{t,s}) + \mathbb G(B_s^c)}\\
& = \frac{\mathbb G(B_{t/2}) + \mathbb G(M_{t,s})}{(\mathbb G(B_{t/2}) +\mathbb G(M_{t,s})) \big(1+ \frac{\mathbb G(B_s^c)}{\mathbb G(B_{t/2}) +\mathbb G(M_{t,s})}\big)} \to 1,
\end{align*}
again with high probability, which is what we wanted to show.

\begin{remark}\normalfont \label{logcn}
If the map $w$ is globally convex, say $w(s)= Ts^2/2$ for all $s>0$, then the `small enough' choice of $s$ after (\ref{bal2a}) is still possible but then depends on the global coercivity constant $T$, which will prevent the previous contraction argument to work. So while the hitting time lower bound is still valid, we cannot conclude that we are never hitting regions of significant posterior probability. It is here where global log-concavity of the likelihood function helps, as it enforces a certain `uniform' spread of the posterior across its support via a global coercivity constant $T$. In contrast the above example of $w$ is not convex, rather it is very spiked on $(0,t/2)$ and then ``flattens out''.  
\end{remark}

\subsection{Proof of Theorem \ref{a}}\label{a-pf}

The proof of Theorem \ref{a} proceeds along the same lines as the one of Theorem~\ref{b}, with scaling $t,L,\rho,s,\eta$ constant in $N$, corresponding to $b=0$ in $N^{-b}$, and replacing the volumetric Lemma \ref{asb} by the following basic result.

\begin{lemma}
Let $\theta\sim \mcN(0,I_D/D)$. Let $a\in (0,1/2)$. Then for all $D\ge D_0(a)$ large enough,
\begin{equation}\label{pm}
- \frac{1}{D} \log \Pi(\|\theta\|_{\R^D} \le z) \ge \frac 12  \big(\frac{z^2}2-\log z-\frac 12\big),~~~ \text{any } z\in (0,1-a).
\end{equation}
\end{lemma}

A proof of (\ref{pm}) is sketched in Appendix \ref{iso-small-ball}.
As a consequence of the previous lemma
\[ \frac 1N \log \Pi(\Theta_{s,\eta})\le \frac 1N\log \Pi(B_{2s})\le \frac \kappa 2 (\log 2s -2s^2+\frac 12). \] Moreover, to lower bound $\Pi(\Theta_{2/3,\eps})$, we choose $\eps>2/3$. Then, using Theorem 2.5.7 in \cite{GN16} as well as $\EE \|\theta\|\le \EE (\|\theta\|^2)^{1/2}=1$, and then also (\ref{pm}) with $z=2/3$, we obtain that 
\begin{align*}
    \Pi(\Theta_{2/3,\eps}) &\ge \Pi(|\|\theta\|_{\R^D}-1 | \le 1/3 )\\
    &\ge 1- \Pi(\|\theta\|_{\R^D} \ge \EE \|\theta\|_{\R^D} +1/3) -\Pi(\|\theta\|_{\R^D} \le  2/3 )\\
    &\ge 1- \exp(-D/18) - \exp(-cD),
\end{align*}
for some fixed constant $c>0$ given by (\ref{pm}), whence $\Pi(\Theta_{2/3,\eps}) \to 1$ and also $N^{-1}\log \Pi(\Theta_{2/3,\eps})\to 0$. Therefore, the key inequality (\ref{bal2}) with $\sigma=2/3,~\nu=1$ holds whenever we choose $s=s_0$ small enough such that 
$$-\log 2s_0 > 2\kappa^{-1}\big[  2 + \frac{\rho}{2}(s_0-2/3-\eps) \big] + 2s_0^2+ \frac{1}{2}.$$

The rest of the detailed derivations follow the same pattern as in the proof of Theorem \ref{b} and are left to the reader, including verification of (\ref{L-hypo}) via an application of Theorem 2.5.7 in \cite{GN16}. In particular, the proof of part \textbf{iii)} follows the same arguments (suppressing the $N^{-b}$ scaling everywhere) as in Theorem \ref{b}.

\subsection{Proofs for Section \ref{cst}}\label{algo-pfs}

In this section, we prove the results of Section~\ref{cst} which detail the consequences of the general Theorems~\ref{a} and \ref{b} for practical MCMC algorithms.

\subsubsection{Proofs for pCN}

Theorem \ref{pcn} is proved by verifying the Assumption \ref{ass:step-eta} for suitable choices of $\eta$ and $L$, and for $c_0=\kappa/2 >0$.

\begin{lemma}\label{pcn-steps}
	Let $\mathcal P_N$ denote the transition kernel of pCN from (\ref{pcn-algo}) with parameter $\beta>0$.
	
	\textbf{i)} Suppose $\Pi=\mcN(0,I_D/D)$ as in Theorem \ref{a}, and let $L,\eta>0$. Then for all $\beta \le \min \{1/2, \eta/4L, \eta^2/64 \}$ and all $D\ge 1$, we have (with $P_0^N$-probability 1)
	\[ \sup_{\theta\in B_L}  \mathcal P_N(\theta,\{\vartheta: \| \theta-\vartheta \|_{\R^D}\ge \eta/2 \})\le e^{-D/2}.  \]
	
	\textbf{ii)} Suppose $\Pi=\mcN(0,\Sigma_\alpha)$ as in Theorem \ref{b}, and let $L,\eta>0$. There exists some $c>0$ such that for all $\beta \le \min \{1/2, \eta/4L, c\eta^2/D\}$ and all $D\ge 1$, we have (with $P_0^N$-probability 1)
	\[ \sup_{\theta\in B_L}  \mathcal P_N(\theta,\{\vartheta: \| \theta-\vartheta \|_{\R^D}\ge \eta/2 \})\le e^{-D/2}.  \]
	
\end{lemma}

\begin{proof}[Proof of Lemma~\ref{pcn-steps} --] We begin with the proof of part \textbf{ii)}. Let $\|\vartheta_k\|_{\R^D}\le L$. Then using the definition of pCN and that $|\sqrt{1-\beta}-1|\le \beta$ for any $\beta \in [0,1]$ (Taylor expanding $\sqrt \cdot$ around $1$), we obtain that for any $\beta \le \min \{ 1/2, \eta/4L\}$,
	\begin{align*}
		\Pr( \| \vartheta_{k+1} -\vartheta_{k}\|_{\R^D}\ge \eta/ 2)&\le \Pr( \| p_{k+1} -\vartheta_{k}\|_{\R^D}\ge \eta/ 2) \\
		&\le \Pr( \|(\sqrt{1-\beta} -1) \vartheta_k \|_{\R^D} + \sqrt \beta \| \xi_k\|_{\R^D} \ge \eta/ 2)\\
		&\le \Pr( \| \xi_k \|_{\R^D} \ge (\eta/2 - \beta L )/ \sqrt \beta)\\
		&\le \Pr( \| \xi_k \|_{\R^D} \ge \frac {\eta}{4\sqrt \beta})\\
		&= \Pr( \| \xi_k \|_{\R^D}- \EE\| \xi_k \|_{\R^D} \ge \frac {\eta}{4\sqrt \beta}- \EE\| \xi_k \|_{\R^D} ).
	\end{align*}
	The variables $\xi_k$ are equal in law to a vector with components $(i^{-\alpha/d} g_i: i \le D)$ for $g_i$ iid $N(0,1)$ and hence $\EE\| \xi_k \|_{\R^D} \le (\EE\| \xi_k \|^2_{\R^D})^{1/2} \le C(\alpha,d)<\infty$ for $\alpha>d/2$. Then, for $\beta\le c \eta^2/D$ with some sufficiently small $c>0$ (noting that then also $\beta\le c \eta^2$), it holds that 
	\begin{align} \label{svenforgot}
	    \Pr( \| \vartheta_{k+1} -\vartheta_{k}\|_{\R^D}\ge \eta/ 2)&\le \Pr( \| \xi_k \|_{\R^D}- \EE\| \xi_k \|_{\R^D} \ge \frac{\eta}{8\sqrt \beta})\le \exp\big(-\frac{\eta^2}{64\beta}\big)\le \exp\big(-D/2\big),
	\end{align}
using, e.g. Theorem 2.5.8 in \cite{GN16} (and representing the $\|\cdot\|_{\R^D}$-norm by duality as a supremum). This completes the proof of part \textbf{ii)}.

	\smallskip 
	
    The proof of part \textbf{i)} is similar, albeit simpler, whence we leave some details to the reader. Arguing similarly as before, we obtain that for any $\beta \le \min \{ 1/2, \eta/64L\}$,
	\begin{align*}
		\Pr( \| \vartheta_{k+1} -\vartheta_{k}\|_{\R^D}\ge \eta/ 2)&\le \Pr( \| \xi_k \|_{\R^D} \ge (\eta/2 - \beta L )/ \sqrt \beta)\le \Pr( \| g_k \|_{\R^D} \ge \frac {\eta \sqrt D}{4\sqrt \beta}),
	\end{align*}
	where $g_k$ is a $\mcN(0,I_D)$ random variable. The latter probability is bounded by a standard deviation inequality for Gaussians, see, e.g.\ Theorem 2.5.7 in \cite{GN16}. Indeed, noting that $\EE \|\xi_k\|_{\R^D}\le (\EE [\|\xi_k\|_{\R^D}^2])^{1/2} = \sqrt D$, and that the one-dimensional variances satisfy $\EE \langle g_k, v\rangle^2 = \|v\|_{\R^D}^2=1 $ for any $\|v\|_{\R^D}=1$, we obtain
	\begin{align*}
		\Pr( \| g_k \|_{\R^D} \ge \frac {\eta \sqrt D}{4\sqrt \beta})& \le  \Pr\big( \big| \| \xi_k \|_{\R^D}- \EE \| \xi_k \|_{\R^D}\big| \ge \sqrt D (\frac \eta{4\sqrt \beta} - 1)\big)\\
		&\le \exp\Big( -\frac D2  (\frac{\eta}{4\sqrt \beta}-1)^2 \Big)\le \exp\Big( -\frac D2 \Big).
	\end{align*}
\end{proof}

\begin{proof}[Proof of Theorem \ref{pcn} --]
We begin with part \textbf{ii)}. Let $s_b$ be as in Theorem \ref{b} and set $\eta=\eta_N=s_bN^{-b}/2$ as well as $L=L_N C\sqrt N$, where $C$ is as in Theorem \ref{b}. With those choices, Lemma \ref{pcn-steps} \textbf{ii)} implies that Assumption \ref{ass:step-eta} is fulfilled with $c_0=\kappa/2$, so long as $\beta$ satisfies
\[ \beta \le \min\Big\{ \frac{1}{2}, \frac{s_bN^{-b}}{8C\sqrt N}, \frac{cs_b^2N^{-2b}}{4D} \Big\} \lesssim N^{-2b}D^{-1}\simeq N^{-1-2b}. \]
Hence, the desired result immediately follows from an application of Theorem \ref{b} \textbf{iv)}.

Part \textbf{i)} of Theorem \ref{pcn} similarly follows from verifying Assumption \ref{ass:step-eta} with $s\in (0,1/3)$, $L$ from Theorem \ref{a}, $\eta= s/2$ and for small enough $\beta <c_1$ (with $c_1$ determined by Lemma \ref{pcn-steps} \textbf{i)}), and subsequently applying Theorem \ref{a} \textbf{iv)}.
\end{proof}

\subsubsection{Proofs for MALA}

Theorem \ref{MALA} is proved by verifying the hypotheses of Theorems \ref{a} and \ref{b} respectively. A key difference between pCN and MALA is that the proposal kernels for MALA, not just its acceptance probabilities, depend on the data $Z^{(N)}$ itself. Again, we begin by examining part \textbf{ii)} which regards $N(0,\Sigma_\alpha)$ priors.

\smallskip 

\begin{proof}[Proof of Theorem \ref{MALA}, part ii)]
	We begin by deriving a bound for the gradient $\nabla \log \pi(\cdot | Z^{(N)})$. For Lebesgue-a.e. $\theta \in \R^D$, recalling that $\mathrm{vol}(\mathcal X)=1$, we have that
	\begin{align*}
	\mathrm E_{0}^N[\nabla \ell_N(\theta)] &= - \frac N2 w'(\|\theta\|)\frac{\theta}{\| \theta\| } \|g\|_{L^2}^2,\\
	\nabla \ell_N(\theta) &= \frac 12\sum_{i=1}^N \Big(\eps_i - \sqrt{w(\|\theta\|) }g(X_i)\Big)\frac{w'(\|\theta\|)}{2\sqrt{w(\|\theta\|)}}\frac{\theta}{\| \theta\| }g(X_i),\\
	&= \frac{w'(\|\theta\|)}{4\sqrt{w(\|\theta\|)}}\frac{\theta}{\| \theta\| }\sum_{i=1}^N \eps_i g(X_i) - \frac{w'(\|\theta\|)}{4} \frac{\theta}{\|\theta\|} \sum_{i=1}^Ng^2(X_i).
	\end{align*}
	For any $r\in (0,t/2)\cup (t/2,t)\cup (t,L)\cup (L,\infty)$, recalling the choices for $T,t,\rho$ in (\ref{choices}) we see that
	\begin{align}\label{w-est}
        \nonumber
	    \frac{w'(r)}{\sqrt{w(r)}} &=\frac{8Tr}{2Tr} 1_{(0,t/2)}(r) + \frac{T}{\sqrt{w(r)}} 1_{(t/2,t)}(r)+\frac{\rho}{\sqrt{w(r)}} 1_{(t,L)}(r), \\
	    &\lesssim  1 + N^{b} + 1,
	\end{align}
	where, to bound the second and third term, we used that $\sqrt {w(r)}\ge Tt = t_bT_b>0$ is bounded away from zero uniformly in $N$ on  $(t/2,\infty)$. Similarly, we have 
	\[  \| w'\|_{\infty} \le Tt/2 + T + \rho \lesssim N^b. \]
	Combining the above and using Chebyshev's inequality, it follows that
	\begin{align*}
	    \sup_{\theta\in \R^D} \|\nabla \ell_N(\theta)\|_{\R^D} &\lesssim N^b \Big( \Big|\sum_{i=1}^N \eps_i g(X_i)\Big| + \sum_{i=1}^Ng^2(X_i) \Big)\\
	    &\le N^b \Big(\|g\|_\infty \Big|\sum_{i=1}^N \eps_i\Big| + \sum_{i=1}^N(g^2(X_i)- \|g\|_{L^2}^2) + N\|g\|_{L^2}^2 \Big)\\
	    &\le N^b ( \mcO_{P}(\sqrt N)+\mcO(N) )\\
	    &= \mcO(N^{1+b}) + \smallO(N^{1+b}).
	\end{align*}
	Thus, the event
	\[A:=\{ \sup_{\theta\in \R^D} \|\nabla \ell_N(\theta)\|_{\R^D} \le C' N^{1+b} \}, \]
	for some large enough $C'>0$, has probability $P_0^N(A)\to 1$ as $N\to\infty$. We also verify that 
	\begin{align}\label{prior-grad}
	    \nabla \log \pi(\theta) =-\frac 12  \nabla \theta^T\Sigma_\alpha^{-1} \theta =-\Sigma_\alpha^{-1}\theta,
	\end{align}
    so that with $L=L_N=C\sqrt N$ (for $C$ as in Theorem \ref{b}) and recalling that $\Sigma_\alpha= \textnormal{diag}(1,\dots,D^{-2\alpha})$, we obtain
    \[ \sup_{\|\theta\|\le L} \|\nabla \log \pi(\theta) \|_{\R^D}= \sup _{\|\theta\|\le L}\|\Sigma_\alpha^{-1}\theta\|_{\R^D}\lesssim D^{2\alpha}\sqrt N\simeq N^{2\alpha+1}. \]
    
    Now, let $s_b,C$ be as in Theorem \ref{b} and set $\eta=\eta_N= \frac 12 s_bN^{-b}$ (note that this is a permissible choice in Theorem \ref{b}) as well as $L=L_N=C\sqrt N$. Furthermore, for a small enough constant $c>0$, let $\gamma \le cN^{-1-2\alpha-b}$. Then since $\alpha>b$, we also have that
    \begin{align}\label{beaurocracy}
        \gamma\lesssim \min \{ N^{-1-2\alpha-b}, N^{-1-2b}, N^{-1/2-b}\}.
    \end{align}
    Hence, on the event $A$ and whenever $\|\theta\|_{\R^D}\le L$, 
    \begin{align*}
    \gamma \| \nabla \log\pi(\vartheta_k | Z^{(N)}) \|_{\R^D} \lesssim  \gamma (N^{1+b}+ N^{1+2\alpha})\lesssim \eta.
    \end{align*}
    Using this, (\ref{beaurocracy}) and choosing $c>0$ small enough, conditional on the event $A$ the probability $\Pr(\cdot)$ under the Markov chain satisfies
	\begin{align*}
	\Pr\big(\|  p_{k+1}- \vartheta_k \|  \ge \eta /2 \big) &\le \Pr\big(	\gamma \| \nabla \log\pi(\vartheta_k | Z^{(N)}) \|_{\R^D} \ge \eta /4 \big) + \Pr\big(	\sqrt {2\gamma} \|\xi_{k+1} \|_{\R^D}\ge  \eta /4 \big) \\
	&\le \Pr\big(\|\xi_{k+1} \|_{\R^D} \ge  \frac{\eta}{4\sqrt {2\gamma} } \big)\\
	&\le \Pr\big(\|\xi_{k+1} \|_{\R^D}- \EE\|\xi_{k+1} \|_{\R^D} \ge  \sqrt N \big)\le \exp\big(-\frac{N}{2}\big),
	\end{align*}
 where the last inequality is proved as in (\ref{svenforgot}) above, using Theorem 2.5.8 in \cite{GN16}. Thus, Assumption \ref{ass:step-eta} is satisfied with $c_0=1$ and the proof is complete.
\end{proof}

\begin{proof}[Proof of Theorem \ref{MALA}, part i) --]
	The proof of part i) proceeds along the same lines, except that (\ref{w-est}) and (\ref{prior-grad}) are replaced with the bound
	\[ \Big\|\frac{w'}{\sqrt w}\Big\|_\infty + \|w'\|_\infty < C, \]
	for some constant $C$ independent of $N$, as well as the bound
	\[ \nabla \log \pi(\theta) =-\frac D2  \nabla \|\theta\|^2 =-D\theta,~~~~  \sup_{\|\theta\|\le L} \|\nabla \log \pi(\theta)\|_{\R^D} \simeq NL. \]
	Then letting $s\in (0,1/3)$ and $L>0$ be as in Theorem \ref{a}, and fixing an arbitrary $\eta \in (0,s/2)$, the above implies that for sufficiently small constant $c>0$ and for any $\gamma \le c/N$, it holds that
	\begin{align*}
	\Pr\big(	\|  p_{k+1}- \vartheta_k \|  \ge \eta /2 \big) &\le \Pr\big(	\gamma \| \nabla \log\pi(\vartheta_k | Z^{(N)}) \|_{\R^D} \ge \eta /4 \big) + \Pr\big(	\sqrt {2\gamma} \xi_{k+1} \ge  \eta /4 \big) \\
	&\le \Pr\big(	\xi_{k+1} \ge  \frac{\eta}{4\sqrt {2\gamma} } \big)\\
	&\le  \Pr\big(	\xi_{k+1} \ge  \frac{\eta \sqrt{\kappa D}}{4\sqrt{2c}} \big).
	\end{align*}
	Thus, choosing $c>0$ small enough and arguing exactly as in the last step of the proof of Theorem \ref{pcn}, part \textbf{i)}, Assumption \ref{ass:step-eta} is satisfied with $c_0=1$ and the proof is complete.
\end{proof}

\appendix

\section{Proofs of Section~\ref{sec:spiked_tensor}}\label{sec_app:proofs_spiked_tensor}

\begin{proof}[Proof of Corollary~\ref{cor:posterior_contraction} --]
    We fix $K = 1$ and place ourselves under the event of Proposition~\ref{prop:posterior_contraction_technical}, and we denote $s = s(\lambda)$ and $t = t(\lambda)$.
    We can decompose, since $\Pi[\shire_s|Y] = 1-\Pi[\doom_s|Y]$:
    \begin{equation*}
        \Pi(\doom_s|Y) = \Big[1 + \frac{\Pi(\shire_s|Y)}{\Pi(\doom_s|Y)}\Big]^{-1}. 
    \end{equation*}
    Moreover, $\Pi(\shire_s|Y)/\Pi(\doom_s|Y) = \Pi(\shire_s|Y)/[\Pi(\doom_t|Y) + \Pi(\mordor_{s,t}|Y)] \leq \Pi(\shire_s|Y)/\Pi(\doom_t|Y)$.
    Using Proposition~\ref{prop:posterior_contraction_technical}, for $n \geq n_0(\lambda,Y)$ we have 
    $\Pi(\shire_s|Y)/\Pi(\doom_s|Y) \leq \exp\{-n\}$. Therefore $\Pi[\doom_s|Y] \geq (1+\exp\{-n\})^{-1}$, which ends the proof.
\end{proof}

\noindent
\begin{proof}
The rest of this section is devoted to proving Proposition~\ref{prop:posterior_contraction_technical}.
We use a uniform bound on the injective norm of Gaussian tensors:
\begin{lemma}\label{Lemma:injectivenorm}
For all $p \geq 3$ there exists a constant $C_p$, such that:
\begin{equation}\label{eq:injectivenorm}
\limsup_{n \to \infty} \Big\{n^{-1/2}\max_{x\in \bbS^{n-1}} | \langle x^{\otimes p},Z \rangle | \Big\}\leq C_p, \hspace{0.5cm} \textrm{almost surely.}
\end{equation}
\end{lemma}
\noindent
This lemma is a very crude version of much finer results: in particular the exact value of the constant $\mu_p$ such that (w.h.p.) 
$\max_{x\in \bbS^{n-1}} | \langle x^{\otimes p},Z \rangle | = \sqrt{n} \mu_p (1 + \smallO_n(1))$ has been first computed 
non-rigorously in \cite{crisanti1992sphericalp}, and proven in full generality in \cite{subag2017complexity} (see also 
discussions in \cite{richard2014statistical,perry2020statistical}).
In the rest of this proof, we assume to have conditioned on eq.~\eqref{eq:injectivenorm}.
For any $0 \leq s < t \leq 1$, we have for $n \geq n_0(Y)$:
\begin{align}
    \nonumber
    \frac{\Pi(\shire_s|Y)}{\Pi(\doom_t|Y)} &= \frac{\int_{\shire_s} \exp(\ell_Y(x)) \rd\Pi(x)}{\int_{\doom_t} \exp(\ell_Y(x)) \rd\Pi(x)}
    \leq e^{n \lambda C_p} \frac{\int_{\shire_s} \exp\Big(\frac{n}{2} \lambda^2 \langle x, x_0 \rangle^p \Big) \rd\Pi(x)}{\int_{\doom_t}  \exp\Big(\frac{n}{2} \lambda^2 \langle x, x_0 \rangle^p \Big) \rd\Pi(x)}, \\
    &\leq \exp\Big(n \lambda C_p + \frac{n\lambda^2}{2} [s^p - t^p]\Big) \frac{\Pi(\shire_s)}{\Pi(\doom_t)}.
    \label{eq:ratio_posteriors}
\end{align}
We upper bound $\Pi(\shire_s) \leq \Pi(\bbS^{n-1}) = 1$. To lower bound $\Pi(\doom_t)$, we use the elementary fact (which is easy to prove using spherical coordinates):
\begin{align}\label{eq:U_bt_bessel}
   \Pi(\doom_t) &= c_p I_{(1-t)/2}[(n-1)/2,(n-1)/2],
\end{align}
in which $I_x(a,b) = \int_0^x u^{a-1} (1-u)^{b-1} \rd u / \int_0^1 u^{a-1} (1-u)^{b-1} \rd u$ is the incomplete beta function, 
and $c_p = 1$ for odd $p$ and $c_p = 2$ for even $p$. 
It is then elementary analysis (cf.\ e.g.\ \cite{perry2020statistical}) that 
\begin{align}\label{eq:asymptotics_U_doomt}
    \lim_{n \to \infty} \frac{1}{n} \log \Pi(\doom_t) &= \frac{1}{2} \log (1-t^2),
\end{align}
uniformly in $t \in [0,1)$.
Coming back to eq.~\eqref{eq:ratio_posteriors}, this implies that we have, for any $s < t < 1$:
\begin{align}\label{eq:ratio_posteriors_2}
    \limsup_{n \to \infty} \frac{1}{n} \log \frac{\Pi(\shire_s|Y)}{\Pi(\doom_t|Y)} &\leq \lambda C_p + \frac{\lambda^2}{2} [s^p - t^p] - \frac{1}{2} \log (1-t^2).
\end{align}
Let $K > 0$. 
It is then elementary to see that it is possible to construct
$0 \leq s(\lambda) < t(\lambda) < 1$ with $\lim_{\lambda \to \infty} \{s(\lambda),t(\lambda)\} = 1$, 
and such that the right-hand-side of eq.~\eqref{eq:ratio_posteriors_2} becomes smaller than $-K$ as $\lambda \to \infty$.
\end{proof}

\section{Small ball estimates for isotropic Gaussians}\label{iso-small-ball}

Let $\Pi = \mcN(0, \Id_D/D)$.
In this section, we prove eq.~\eqref{pm}, more precisely we show:
\begin{lemma}\label{lemma:lower_bound_smallball_isotropic}
    Let $a \in (0,1)$. Then for all $D \geq D_0(a)$ large enough, one has for all $z \in (0, 1-a)$: 
    \begin{align}
        - \frac{1}{D} \log \Pi (\|\theta\|_2 \leq z) \geq \frac{1}{2} \Big(\frac{z^2}{2} - \log z - \frac{1}{2}\Big).
    \end{align}
\end{lemma}
\begin{proof}[Proof of Lemma~\ref{lemma:lower_bound_smallball_isotropic} --]
Let $f(x) = -x^2/2 + \log x + 1/2$, so that $f$ reaches its maximum in $x = 1$, with $f(1) = 0$.
By decomposition into spherical coordinates and isotropy of the Gaussian measure, one has directly: 
\begin{align}
    \Pi (\|\theta\|_2 \leq z) &= \frac{\mathrm{vol}(\bbS^{D-1})}{(2\pi/D)^{D/2}} \int_0^z \rd r \, e^{-\frac{D r^2}{2} + (D-1) \log r}.
\end{align}
Recall that $\mathrm{vol}(\bbS^{D-1}) = 2 \pi^{D/2} / \Gamma(D/2)$, so one reaches easily: 
\begin{align}
    c_D = \frac{1}{D} \log \frac{\mathrm{vol}(\bbS^{D-1})}{(2\pi/D)^{D/2}} - \frac{1}{2} &= \frac{\log D}{2 D} + \mcO(1/D).    
\end{align}
In particular, one has for all $D$ large enough (not depending on $z$): 
\begin{align}
    \frac{1}{D} \log\Pi (\|\theta\|_2 \leq z) &\leq \frac{1}{D} \log\int_0^z \rd r \, e^{-\frac{r^2}{2} + (D-1) f(r)} + c_D.
\end{align}
Since $f$ is increasing on $(0,1)$, we have for large enough $D$:
\begin{align}
    \frac{1}{D} \log\Pi (\|\theta\|_2 \leq z) &\leq \Big(1 - \frac{1}{D}\Big) f(z) + c_D + \frac{1}{D} \log\int_0^\infty \rd r \, e^{-r^2/2}, \\ 
    &\leq \Big(1 - \frac{1}{D}\Big) f(z)  + \frac{\log D}{D} .
\end{align}
Since $f(1-a) < 0$, let $D \geq D_0(a)$ large enough such that $f(1-a) \leq - 2 \log D / (D-2)$. 
Then for all $z \leq 1-a$, one has $f(z) \leq - 2 \log D / (D-2)$.
Plugging it in the inequality above, we reach that for all $z \in (0,1-a)$:
\begin{align}
    \frac{1}{D} \log\Pi (\|\theta\|_2 \leq z) &\leq \frac{1}{2}f(z).
\end{align}

\end{proof}

\printbibliography

\end{document}